\documentclass[a4paper, 11pt]{article}
\usepackage{mathrsfs}
\usepackage{amssymb,amsmath,amsfonts,amsthm}
\usepackage{graphicx,subfigure,epstopdf}
\usepackage[usenames]{color}
\usepackage{url}
\usepackage{algorithm,algorithmic}
\usepackage{dsfont,verbatim}
\usepackage[colorlinks,linktocpage,linkcolor=blue]{hyperref}
\usepackage[margin=1.08in]{geometry}
\usepackage{enumerate,enumitem}
\usepackage[normalem]{ulem}
\usepackage{cite}
\usepackage{bm}
\allowdisplaybreaks

\definecolor{darkred}{rgb}{.7,0,0}

\definecolor{green}{rgb}{0,0.7,0}

\newcommand{\nn}{\nonumber}

\newtheoremstyle{thmm}{1.5ex plus 1ex minus .2ex}{1.5ex plus 1ex minus.2ex}{\rmfamily}{}{\bfseries}{}{1em}{} \theoremstyle{thmm}
\newtheorem{theorem}{Theorem}[section]
\newtheorem{lemma}{Lemma}[section]

\newtheorem{remark}{Remark}[section]
\renewenvironment{proof}[1][Proof]{\noindent\textit{#1. }
}{\hfill$\square$}

\newcommand{\vertiii}[1]
{{\left\vert\kern-0.25ex\left
\vert\kern-0.25ex\left\vert #1
    \right\vert\kern-0.25ex\right
\vert\kern-0.25ex\right\vert}}

\def\d{{\rm d}}

\def\R{\mathbb{R}}

\def\N{\mathbb{N}}

\def\mat{\partial_t^{\bullet}}

\def\v{{\bf v}}

\def\d{{\mathrm d}}

\def\R{{\mathbb R}}

\def\rg{{H}}

\def\Re{\text{\rm Re}}

\def\ls{\Gamma_{h}} 
\def\mdt{\partial_t^\bullet} 
\def\mdth{\partial_{h,t}^\bullet} 
\def\Gaht{\Gamma_h(t)}
\def\Gahs{\Gamma_h(s)}
\def\Gat{\Gamma(t)}

\def\Gah{\Gamma_h}
\def\ll{\|\hspace{-1pt}|}

\newcommand{\dimsurf}{m}

\newcommand{\vphi}{\varphi}

\newcommand\blfootnote[1]{%
  \begingroup
  \renewcommand\thefootnote{}\footnote{#1}%
  \addtocounter{footnote}{-1}%
  \endgroup
}

\title{\Large\bf Maximal regularity of evolving FEMs \\ 
for parabolic equations on an evolving surface}
\author{\normalsize 
Genming Bai$^\dagger$
\and
Bal\'azs Kov\'acs \thanks{Institute of Mathematics, Paderborn University, Warburgerstr.~100, 33098 Paderborn, Germany, Email address: balazs.kovacs@math.upb.de}
\and
Buyang Li\thanks{Department of Applied Mathematics, The Hong Kong Polytechnic University, Hung Hom, Kowloon, Hong Kong. 
Email address: genming.bai@connect.polyu.hk, buyang.li@polyu.edu.hk }
}

\date{}

\begin{document}

\maketitle

\vspace{-10pt}

\begin{abstract}
In this paper, we prove that spatially semi-discrete evolving finite element method for parabolic equations on a given evolving hypersurface of arbitrary dimensions preserves the maximal $L^p$-regularity at the discrete level. We first establish the results on a stationary surface and then extend them, via a perturbation argument, to the case where the underlying surface is evolving under a prescribed velocity field. The proof combines techniques in evolving finite element method, properties of Green's functions on (discretised) closed surfaces, and local energy estimates for finite element methods. 
\\ 

\end{abstract}

\blfootnote{This work is partially supported by a Hong Kong RGC grant (project no. 15300519). }

\vspace{-20pt}

\setlength\abovedisplayskip{4pt}
\setlength\belowdisplayskip{4pt}

\tableofcontents

\section{Introduction}
\label{section:introduction}
\setcounter{equation}{0}

\textbf{Overview of maximal regularity results in flat domains.}
Let $\varOmega\subset\R^d$ be a flat, polyhedral of dimensions $d\in\{2,3\}$ or smooth domain of arbitrary dimensions, and consider the initial and boundary value problem for a linear parabolic partial differential equation (PDE)
\begin{align} 
\label{PDE1} 
\left\{\begin{array}{ll}
\displaystyle
\frac{\partial u(t,x)}{\partial t}- \sum_{i,j=1}^d\frac{\partial}{\partial x_i}\bigg(a_{ij}(x)\frac{\partial u(t,x)}{\partial x_j}\bigg) = f(t,x)
&\mbox{for}\,\,(t,x)\in \R_+\times \varOmega,\\[10pt]
u(t,x)=0 
&\mbox{for}\,\,(t,x)\in \R_+\times\partial\varOmega ,\\[7pt]
u(0,x)=u_0(x)
&\mbox{for}\,\,x\in \varOmega ,
\end{array}\right. 
\end{align} 
where $a_{ij}=a_{ji}$ are real-valued bounded measurable functions satisfying the following uniform ellipticity condition for some constant $\lambda>0$:
\begin{equation} \label{Ellipticitiy}
\lambda^{-1}|\xi|^2 
\le \mbox{$\sum_{i,j=1}^d$} a_{ij}(x) \xi_i\xi_j 
\le \lambda|\xi|^2,\quad\forall\, \xi=(\xi_1,\dotsc,\xi_d)\in\R^d,\,\,\,\forall\, x\in\varOmega .
\end{equation} 
Under condition \eqref{Ellipticitiy}, the elliptic partial differential operator $A=\sum_{i,j=1}^d\frac{\partial}{\partial x_i}\big(a_{ij}(x)\frac{\partial }{\partial x_j}\big)$ generates a bounded analytic semigroup on $L^q(\varOmega)$, $1<q<\infty$, and the solution of \eqref{PDE1} possesses the following maximal $L^p$-regularity in $L^q(\varOmega)$: 
\begin{align}\label{MaxLpReg}
\hspace{-5pt}
\|\partial_tu\|_{L^p(\R_+;L^q(\varOmega))}
+\|A u\|_{L^p(\R_+;L^q(\varOmega))}\leq C
\|f\|_{L^p(\R_+;L^q(\varOmega))}, 
\quad \mbox{if}\,\,\, u_0=0 ,\,\, 1<p,q<\infty ,
\end{align} 
which is an important mathematical tool in studying the well-posedness and regularity of solutions of nonlinear parabolic PDEs; see \cite{Amann1995,ClementPruss1992,LiYang2015,KW04,Lunardi12,Pruss01}.

Analogously, denoting by $A_h$ the finite element approximation of the elliptic operator $A$ on a finite element subspace $S_h\subset H^1_0(\varOmega)$, i.e.,  
\begin{align}\label{Discrete-Laplacian}
(A_h\phi_h,\varphi_h)=-(a_{ij}\nabla \phi_h,\nabla\varphi_h),\quad \forall\, \phi_h, \varphi_h\in S_h, 
\end{align} 
it is known that the semi-discrete finite element solutions given by  
\begin{align}\label{FEMEq0}
\left\{\begin{array}{ll}
(\partial_t u_h,v_h)+\sum_{i,j=1}^d (a_{ij} \partial_ju_h ,\partial_iv_h)=(f,v_h),
&\forall\, v_h\in S_h ,\,\,
\forall\, t\in(0,T),\\[5pt]
u_h(0)=u_{h,0} ,
\end{array}\right.
\end{align} 
satisfies the following spatially discrete $h$-uniform maximal $L^p$-regularity  \cite{Li2015,LiSun2017-MCOM} (with a constant $C>0$ independent of the mesh size $h$):   
\begin{align}
&\|\partial_tu_h\|_{L^p(\R_+;L^q(\varOmega))}
+\|A_hu_h\|_{L^p(\R_+;L^q(\varOmega))}\leq C 
\|f\|_{L^p(\R_+;L^q(\varOmega))} ,   
&&\mbox{if}\,\,\, u_{h,0}=0,\,\,\,
1<p,q<\infty,
\label{MaxLpReg-FEM}
\end{align} 
which has applications in numerical analysis for semilinear parabolic equations with strong nonlinearities \cite{Geissert2007}, and quasi-linear parabolic equations with nonsmooth coefficients \cite{LiSun2015-regularity}.  The spatially discrete maximal $L^p$-regularity results were firstly proved in smooth domains with Neumann boundary condition \cite{Geissert2006,Li2015}, and then extended to polyhedral domains \cite{LiSun2017-MCOM,Li-MCOM-2019} with the Dirichlet boundary condition. The discrete maximal $L^p$-regularity is also closely related (in the techniques of proof) to the maximum-norm stability of finite element solutions of parabolic equations \cite{Leykekhman2004,NitscheWheeler1982,SchatzThomeeWahlbin1980,SchatzThomeeWahlbin1998,ThomeeWahlbin2000}:
\begin{align}\label{maximum-norm}
\|u-u_h\|_{L^\infty(0,T;L^\infty(\Omega))}
\le C\ln(2+1/h)\inf_{\chi_h \in L^\infty(0,T;S_h)}\|u-\chi_h\|_{L^\infty(0,T;L^\infty(\Omega))}.
\end{align} 

The extension of maximal $L^p$-regularity to the time-discrete setting was established for different time discretization methods, including the backward Euler method \cite{AshyralyevPiskarevWeis2002}, discontinuous Galerkin method \cite{leykekhman2017discrete}, $\theta$-schemes \cite{Kem18}, and A-stable multistep and Runge-Kutta methods \cite{KLL16}. All these methods are A-stable. The maximal $L^p$-regularity of A$(\alpha)$-stable backward differentiation formulae (BDF) was established in \cite{Li-IMA-2022}. 
The discrete maximal $L^p$-regularity helps us to control the nonlinearlity as well (see \cite{ALL17,KLL18_focm,KL23}), and besides it enables us to obtain optimal-order $L^p$-norm error estimates without using the inverse inequality (cf.~\cite{KLL16,Li-MCOM-2019,KL23}).

\textbf{Overview of maximal regularity results on surfaces.} 
The maximal $L^p$-regularity of parabolic equations on an \emph{evolving surface} $\Gamma \subset \R^{m+1}$, $m\in \N$, as well as the maximal $L^p$-regularity of time discretizations on an evolving surface and its application to the convergence analysis of BDF methods for nonlinear PDEs on an evolving surface, was discussed in \cite{KL23}. 

\medskip
\textbf{Semi-discrete maximal regularity results on evolving surfaces.} 
However, since the maximal $L^p$-regularity of spatial discretizations for parabolic equations on an evolving surface is still missing, only semi-discretization in time were considered in \cite{KL23}. The aim of this article is to fill in this gap, by establishing spatially discrete maximal $L^p$-regularity of isoparametric finite element methods (FEMs) for parabolic equations on an evolving surface which is approximated by quasi-uniform curved triangles.

In order to prove the discrete maximal $L^p$-regularity for the spatially semi-discrete problems, we combine the techniques developed for evolving surface FEMs and local energy estimates. Firstly, we shall prove the discrete maximal $L^p$-regularity for spatially semi-discrete FEM in \eqref{eq:scheme} on a stationary surface (Theorem \ref{thm:disc_max_reg2}). Then we use a temporal perturbation argument to extend this result to evolving surfaces (Theorem \ref{thm:disc_max_reg}). The discrete maximal $L^p$-regularity results for \eqref{eq:scheme-2} can be obtained analogously by a perturbation argument for the lower-order advection term.

Since we are considering a spatially semi-discrete setting, the underlying smooth surface $\Gamma(t)$ should also be replaced by the finite element surface $\Gamma_h(t)$. The discrepancy between $\Gamma(t)$ and $\Gamma_h(t)$ is the main obstacle in the proof, leading to the following technical difficulties to be addressed: 
\begin{itemize}
	\item The discrete delta functions on $\Gamma(t)$ and $\Gamma_h(t)$ are not simply related by the lift via the distance projection. Indeed, they are correlated via a nonlinear way which stems from the nonlinear relation of $L^2$ projections $P_h(\Gamma(t))$ and $P_h(\Gamma_h(t))$. Therefore, it is necessary for us to obtain the high-order consistency between the discrete delta functions on different surfaces (Lemma \ref{lemma:delta}) in order to ensure the consistency of the corresponding Green's functions, which are indispensable in the used local energy estimate (Lemma~\ref{LocEEst}) and dyadic decomposition argument (Lemma~\ref{LemGm2}).
	\item The discrepancy of $\Gamma(t)$ and $\Gamma_h(t)$ will introduce a bunch of additional geometric perturbation terms in the local energy estimate (Lemma \ref{LocEEst}) and in the dyadic decomposition argument (Lemma \ref{LemGm2}). We need to treat them carefully to make sure that the leading order stability is still available.
	\item In the temporal perturbation argument, we need to develop the norm equivalence of the discrete Laplacian (Lemma \ref{lemma:delta_h} and Remark \ref{rmk:delta_h}). The super-approximation property (cf.~(P3) in Section \ref{sec:hypo}) also plays an important role in the derivation of this equivalence. Besides, as a nature of parametric finite elements, the matrix-valued coefficient $B_h(t,x)$ of the following change of variable 
	\begin{align}
		\int_{\Gaht} \nabla_{\Gaht} u^{-\ell} \nabla_{\Gaht} v^{-\ell}  = \int_{\Gat} B_h(t,\cdot) \nabla_{\Gat} u \nabla_{\Gat} v \notag
	\end{align}
	always has jumps at the edges and thus is discontinuous. To this end, it is desirable as well to construct a globally continuous substitute $\tilde B(t,x)$ according to the definition of the discrete Laplacian.
	\item The norm equivalence of the discrete Laplacian will bring in a lower-order term $\| \nabla u \|_{L^p}$. To control this term by the maximal $L^p$-regularity, we need to use the discrete interpolation inequality (Lemma \ref{lemma:interp}) whose proof greatly relies on the $W^{1,p}$-stability of the Ritz projection. The latter is a consequence of the Green's function estimate on closed manifold ({\cite[Theorem~3.2]{Dem09}}).
\end{itemize}

The article is organized as follows:
In Section~\ref{section:ESFEM}, we introduce the basic notations for evolving surface FEMs, the semi-discrete evolving surface FEMs for parabolic equations on an evolving surface, and the main theoretical results about discrete maximal $L^p$-regularity of evolving surface FEMs. 
In Section~\ref{section:stationary surfaces}, we develop the preliminary results of geometric perturbation estimates, Green's function estimates and local energy estimates on the stationary surface. We will prove the maximal regularity on stationary surface (Theorem \ref{thm:disc_max_reg2}) and on evolving surface via a temporal perturbation argument (Theorem \ref{thm:disc_max_reg}) in Section\ref{section:maximal regularity on stationary surfaces} and Section~\ref{section:perturbation arguments for evolving surfaces}, respectively.
In Appendix~\ref{appendix:A} we present the detailed proof of local energy estimates (Lemma \ref{LocEEst}), which is a technical lemma used in the proof of the main theorems.

\section{Evolving surface FEMs and main results}
\label{section:ESFEM}
\setcounter{equation}{0}

We now introduce the problem to be considered in this paper and briefly recall the standard parametric evolving surface FEM, see \cite{DziukElliott_ESFEM,Dem09,Dziuk2013b,highorderESFEM}.

\subsection{Parabolic equations on an evolving surface}
We assume that the evolution of a closed hypersurface $\Gamma(t) \subset \R^{\dimsurf + 1}$ of arbitrary dimensions is given by a diffeomorphic flow map $X(t,\cdot) \colon \Gamma^0 \to \Gamma(t)$, where $m \in \N$, and $\Gamma^0$ is a smooth $\dimsurf$-dimensional initial hypersurface, with $X(0,\cdot)$ being the identity map on $\Gamma^0$. We assume that $X(t,y)$ is smooth with respect to $(t,y)\in [0,T]\times \Gamma^0$ and the inverse function $X^{-1}(t,x)$ is smooth with respect to $x\in\Gamma(t)$ uniformly for $t\in[0,T]$. 

The material velocity (which is simply called velocity below) and material derivative on the surface are respectively given, for $x=X(y,t) \in \Gamma(t)$ with $y\in\Gamma^0$, by 
\begin{align}
	\label{eq:ODE for positions}
	v(t,x) =&\ \partial_t X(t,y) , \\
	\intertext{and} 
	\nonumber
	\mat u(t,x)= &\ \frac{\d}{\d t} u(t,X(t,y)) .
\end{align} 
Let $\nu$ be the unit outward normal vector to the surface $\Gamma(t)$. We denote by $\nabla_{\Gamma(t)}u$ the tangential gradient of the function $u$, and denote by $\Delta_{\Gamma(t)}u = \nabla_{\Gamma(t)} \cdot \nabla_{\Gamma(t)} u$ the Laplace--Beltrami operator acting on $u$.
For more details on all these basic concepts we refer to \cite{DziukElliott_ESFEM,DeckelnickDziukElliott_acta,Dziuk2013b}, and the references therein. 
An unified abstract theory for evolving mesh FEMs for partial differential equations on evolving domains can be found in \cite{ElliottRanner_unified}.

We consider the following two types of linear parabolic equations on the closed surface $\Gamma(t)$.

(i) A parabolic model problem on a given evolving surface:
\begin{align}
\label{eq:PDE 1}
	\left\{
	\begin{aligned}
		&\mdt u - \Delta_{\Gamma(t)} u = f
		& \qquad &\mbox{on}\,\,\,\Gamma(t) ,\,\,\,\forall\,t\in (0,T),\\
		&u(0,\cdot)= u^0 &&\mbox{on}\,\,\,\Gamma^0 ,
	\end{aligned}
	\right.	
\end{align}
where $\mdt$ is the material derivative associated to the flow velocity $v$. The above model problem arises in the analysis and numerical analysis of mean curvature flow, see \cite{Huisken1984} and \cite{KLL19}, respectively, wherein the following evolution equations of normal vector $\nu$ and mean curvature $H$ play an important role: 
$$
\mdt \nu - \Delta_{\Gamma(t)} \nu = |\nabla_{\Gamma(t)} \nu|^2 \nu
\quad\mbox{and}\quad 
\mdt H - \Delta_{\Gamma(t)} H = |\nabla_{\Gamma(t)} \nu|^2 H . 
$$ 

(ii) The heat equation on a given evolving surface (see, e.g., \cite[equation~(1.1)]{DziukElliott_ESFEM}):
\begin{align}
\label{eq:PDE 2}
	\left\{
	\begin{aligned}
		&\mdt u + u \nabla_{\Gamma(t)} \cdot v - \Delta_{\Gamma(t)} u = f
		& \qquad &\mbox{on}\,\,\,\Gamma(t) ,\,\,\,\forall\,t\in (0,T),\\
		&u(0,\cdot)= u^0 &&\mbox{on}\,\,\,\Gamma^0 .
	\end{aligned}
	\right.	
\end{align}

The regularity of solutions to PDEs on an evolving surface is often characterized by the following Sobolev spaces on a space-time manifold $\mathcal G_T =\bigcup_{t\in(0,T)}(\{t\}\times\Gamma(t))$: 
\begin{align}
	L_t^p(0,T; W^{k,q}(\Gamma(t)))
	&=
	\bigg\{
	w\colon  \mathcal G_T \rightarrow \R : w(t,\cdot) \in W^{k,q}(\Gamma(t)) \mbox{ a.e. } t\in (0,T) , \notag\\
	&\hspace{121pt} t\rightarrow \| w(t,\cdot) \|_{W^{k,q}(\Gamma(t))} \in L^p(0,T) 
	\bigg\} , \\[5pt] 
	W_t^{1,p}(0,T; W^{k,q}(\Gamma(t)))
	&=
	\bigg\{
	w\in L_t^p(0,T; W^{k,q}(\Gamma(t))) : \mdt w \in L_t^p(0,T; W^{k,q}(\Gamma(t)))
	\bigg\} .
\end{align}
The conventional notational convention $H^k(\Gamma(t)) = W^{2,k}(\Gamma(t))$ will also be used. For more details on these spaces we refer to \cite{AlphonseElliottStinner_2015,AlphonseCaetanoDjurdjevacElliottS_2023} (employing a different notation). 

The weak formulation of \eqref{eq:PDE 1} reads as follows. Find $u\in H^1(\mathcal{G}_T) \cap L_t^2(0,T; H^{1}(\Gamma(t)))$ satisfying relation 
\begin{align}
\label{eq:PDE 1 - weak}  
	\int_{\Gamma(t)} \mdt u \, \varphi + \int_{\Gamma(t)} \nabla_{\Gamma(t)} u\cdot \nabla_{\Gamma(t)} \varphi
	= 
	\int_{\Gamma(t)} f \varphi 
\end{align}
for all $\varphi\in H^1(\mathcal{G}_T)$ and almost all $t \in (0,T)$. 

The weak formulation of \eqref{eq:PDE 2} reads as follows (see, e.g., \cite[equation~(1.2)]{DziukElliott_ESFEM}): Find $u\in H^1(\mathcal{G}_T) \cap L_t^2(0,T; H^{1}(\Gamma(t)))$ satisfying relation 
\begin{equation}
\label{eq:PDE 2 - weak}
	\frac{\d}{\d t} \int_{\Gamma(t)} u \, \varphi + \int_{\Gamma(t)} \nabla_{\Gamma(t)} u\cdot \nabla_{\Gamma(t)} \varphi
	= 
	\int_{\Gamma(t)} f \varphi 
\end{equation}
for almost all $t \in (0,T)$ and all $\varphi\in H^1(\mathcal{G}_T)$ with $\mat \vphi = 0$.

\subsection{Evolving surface finite elements}
Let $\Gamma_h(t)$ be the closed and continuous piecewise polynomial surface (of degree $k$) which approximates the smooth surface $\Gamma(t)$ evolving under the prescribed velocity $v$. Additionally, we require that the nodes of $\Gamma_h(t)$ stay on $\Gamma(t)$ and move with the same velocity as $\Gamma(t)$. In other words, the evolution of $\Gamma_h(t)$ is uniquely determined with velocity $v_h = I_h v$. 

Each polynomial element $K$ of $\Gamma_h(t)$ is the image of an element $K^0\subset\Gamma_h(0)$ under the discrete flow map. We denote by $K_{\rm f}^0$ the unique flat element which has the same endpoints as $K^0$, and denote by $F_{K} \colon K_{\rm f}^0\rightarrow K$ the parametrization of $K$, i.e., $F_{K}$ is the unique polynomial of degree $k$ that maps $K_{\rm f}^0$ onto $K$. For a more detailed description of high-order surface triangulations we refer to \cite{Dem09,highorderESFEM}.
The finite element space of degree $k$ on the discrete surface $\Gamma_h(t)$ is defined as 
\begin{align*}
	S_h(\Gamma_h(t)) : = &\ \big\{ \vphi_h\in C(\Gamma_h(t)): \vphi_h\circ F_K \in \mathbb{P}^k(K_{\rm f}^0) \,\,\mbox{for every element}\,\, K\subset\Gamma_h(t) \big\} \\
	= &\ \textnormal{span} \big\{ \phi_1, \dotsc, \phi_N \big\},
\end{align*}
where $\mathbb{P}^k(K_{\rm f}^0)$ denotes the space of polynomials of degree $k \ge 1$ on the flat element $K_{\rm f}^0$, and where $\phi_j$ denote the standard nodal basis functions of $S_h(\Gamma_h(t))$. By construction $S_h(\Gamma_h(t))$ is an isoparametric finite element space.

We denote by $\mdth$ the discrete material derivative associate to the discrete flow $v_h$. The basis functions $\phi_j$ satisfy the following transport property (see \cite[Proposition~5.4]{DziukElliott_ESFEM}):
\begin{equation}
\label{eq:transport property}
	\mdth \phi_j = 0 \qquad \text{for} \quad j=1,\dotsc,N.
\end{equation}


Let $\delta>0$ be a fixed sufficiently small constant such that every point $x$ in the $\delta$-neighbourhood of surface $\Gamma(t)$, denoted by $D_\delta(\Gamma(t))=\{x\in\R^{\dimsurf+1}: {\rm dist}(x,\Gamma(t))\le \delta\}$, 
has a unique distance projection onto $\Gamma(t)$, denoted by $q(t,x)$, satisfying the following relation: 
\begin{align*}
	q(t,x) - x = |x-q(t,x)| \nu(t,q(t,x)) ,
\end{align*}
where $\nu(t,\cdot)$ is the outward unit normal vector to $\Gamma(t)$. It is known that such a constant $\delta$ exists and only depends on the curvature of $\Gamma(t)$ (thus $\delta$ is independent of $t\in[0,T]$, but possibly dependent on $T$). We will use the notation $q^{-1}(t,\cdot)$ to denote the inverse of the bijective map $q(t,\cdot)|_{\Gaht} \colon \Gaht \rightarrow \Gat$. 

We assume that each element $K^0\subset\Gamma_h(0)$ interpolates the smooth initial surface $\Gamma(0)$ and that the parametrization $F_{K^0} \colon K_{\rm f}^0\rightarrow K^0$ is a polynomial of degree at most $k$ with the following property: 
\begin{align}\label{P0}
	& \max_{K^0\subset\Gamma_h(0)} 
	\Big(\|F_{K^0}\|_{W^{k,\infty}(K_{\rm f}^0)} + \|\nabla_{K^0} F_{K^0}^{-1}\|_{L^\infty(K^0)} \Big) 
	\le \kappa_0 ,
\end{align}
where $\kappa_0$ is some constant that is independent of $h$. This property holds for standard parametric finite elements which interpolate the smooth surface $\Gamma(0)$ and guarantees the following optimal-order approximation to $\Gamma(0)$ by $\Gamma_h(0)$ (see \cite{Dem09,Dziuk2013b}): 
\begin{align}
	& \max_{K^0\subset\Gamma_h(0)}  \| q(0) \circ F_{K^0} - F_{K^0} \|_{L^{\infty}(K_{\rm f}^0)} 
	\le Ch^{k+1} .
\end{align}
The projection $q(0,x)$ is well defined for points $x$ in a neighbourhood of $\Gamma(0)$ and therefore well defined on $\Gamma_h(0)$ for sufficiently small mesh size $h$. Since the discrete flow $v_h = I_h v$ is piecewise smooth, at any time $t\in[0,T]$ the optimal-order approximation to $\Gamma(t)$ by $\Gamma_h(t)$ is also available
\begin{align}
	& \max_{K\subset\Gamma_h(t)}  \| q(t) \circ F_{K} - F_{K} \|_{L^{\infty}(K_{\rm f}^0)} 
	\le Ch^{k+1} .
\end{align}


\subsection{Lift}

The \emph{lift} of a finite element function $w_h\in S_h(\Gamma_h(t))$ onto the smooth surface $\Gamma(t)$ is defined as 
$$
 w_h^\ell = w_h\circ (q(t) |_{\Gamma_h(t)})^{-1} ,
$$
see \cite[Section 2.4]{Dem09} and \cite[Section 3.4]{KLL19}. 
The \emph{inverse lift} of a function $f(t,\cdot) \colon \Gamma(t) \to \R$ is denoted by $f^{-\ell}(t,\cdot) \colon \Gamma_h(t) \to \R$, which is the function satisfying that $(f^{-\ell})^\ell = f$.
By lifting the triangulation from $\Gamma_h(t)$ to $\Gamma(t)$, we obtain a triangulation and finite element space on the smooth surface $\Gamma(t)$, i.e., 
$$
	S_h(\Gamma(t)) = \big\{ w_h\in C(\Gamma(t)): w_h\circ q(t) \circ F_K \in \mathbb{P}^k(K_{\rm f}^0) \,\,\mbox{for every element}\,\, K\subset\Gamma_h(t) \big\} .
$$

For the ease of notations, if the specified domain is $\Gamma(t)$ we will also identify $v_h$ as its lift on $\Gamma(t)$. Namely, $v_h^\ell \in S_h(\Gamma(t))$ and $w_h\in S_h(\Gamma_h(t))$ simply mean that $v_h^{-\ell}\in S_h(\Gamma_h(t))$ and $w_h^\ell\in S_h(\Gamma(t))$. 
We will use the explicit superscript $^\ell$ to denote the lift where we want to emphasize its underlying domain.

\subsection{The semi-discrete problems}

We will now formulate the evolving surface finite element semi-discretisations of the two weak formulations in \eqref{eq:PDE 1 - weak} and \eqref{eq:PDE 2 - weak}.

The semi-discretisation of \eqref{eq:PDE 1 - weak} reads: Find $u_h\in C^1_t([0,T];S_h(\Gamma_h(t)))$ such that 
\begin{align}\label{eq:scheme}
	\left\{
	\begin{aligned}
		&(\partial_t u_h(t),\varphi_h)_{\Gaht} + (\nabla_{\Gaht} u_h(t),\nabla_{\Gaht} \varphi_h)_{\Gaht} 
		=
		(f^{-\ell}(t),  \varphi_h)_{\Gaht}
		&&\forall\,t\in(0,T] , \\ 
		&u_h(0)= u_{h,0} &&\mbox{on} \quad \Gamma_h^0 ,
	\end{aligned}
	\right.	
\end{align}
for all $\varphi_h\in S_h(\Gamma_h(t))$,  where $u_h^0 = I_h u(0)$, and $\partial_t u_h = \mdth u_h$ should be interpreted in the nodal sense. Equivalently, \eqref{eq:scheme} can be written into the following strong form
\begin{align}\label{eq:scheme_strong}
	\left\{
	\begin{aligned}
		&\partial_t u_h(t) - \Delta_{\Gaht, h} u_h(t)
		=
		f_h(t)
		&&\forall\,t\in(0,T], \\ 
		&u_h(0)= u_{h,0} &&\mbox{on} \quad \Gamma_h^0 ,
	\end{aligned}
	\right.	
\end{align}
where $f_h(t,x) = P_h(\Gaht) f^{-\ell}(t,x)$ and $P_h(\Gaht) \colon L^2(\Gaht) \rightarrow S_h(\Gaht)$ is the $L^2$ projection on surface $\Gaht$.

Similarly, the semi-discretisation of \eqref{eq:PDE 2 - weak} reads: Find $u_h\in C^1_t([0,T];S_h(\Gamma_h(t)))$ such that 
\begin{align}\label{eq:scheme-2}
	\left\{
	\begin{aligned}
		&\frac{\d}{\d t} (u_h(t),\varphi_h)_{\Gaht} + (\nabla_{\Gaht} u_h(t),\nabla_{\Gaht} \varphi_h)_{\Gaht} 
		=
		(f^{-\ell}(t),  \varphi_h)_{\Gaht}
		&&\forall\,t\in(0,T] , \\ 
		&u_h(0)= u_{h,0} &&\mbox{on} \quad \Gamma_h^0 ,
	\end{aligned}
	\right.	
\end{align}
for all $\varphi_h \in C^1_t([0,T];S_h(\Gamma_h(t)))$ such that $\partial_{h,t}^\bullet \vphi_h = 0$.  

\subsection{Main results: Spatially discrete maximal regularity}
\label{section:main results}


Following \cite{KL23}, we recall that $u \in H^1(\mathcal G_T)$ is a solution of \eqref{eq:PDE 1} if and only if the function $U(t,y) := u(t,X(t,y)) $, where $X\colon \Gamma(0)\rightarrow\Gamma(t)$ is the flow map, defines a solution $U \in H^1(\Gamma(0) \times (0, T))$ of the following weak formulation: 
\begin{align}\label{eq:PDE_pb}
	\int_{\Gamma(0)} a(t,y) \partial_t U(t,y) \psi(y)
	+
	\int_{\Gamma(0)} B(t,y) \nabla_{\Gamma(0)} U(t,y) \cdot \nabla_{\Gamma(0)} \psi(y)
	=
	\int_{\Gamma(0)} a(t,y) F(t,y) \psi(y) ,
\end{align}
for all $\psi \in H^1(\Gamma(0))$ and almost all $t \in [0, T]$, with $F(t,y) := f (X(t,y), t)$ and $U(0,\cdot) = u_0$.

Since the Riemannian metric on the evolving surface is positive definite (see \cite[Appendix]{KL23}), it follows that the functions $a(t,y)$ and $B(t,y)$ satisfy the following estimates:
\begin{align}
	&C^{-1} \leq a(t,y) \leq C, &&\quad\forall y\in \Gamma(0), \forall t\in [0,T], \\
	&C^{-1} |\xi_y|^2 \leq B(t,y)\xi_y\cdot\xi_y \leq C |\xi_y|^2, &&\quad\forall y\in \Gamma(0), \forall \xi_y\in T\Gamma(0)_y, \forall t\in [0,T] ,
\end{align}
where $C$ is some positive constant (depending only on the given flow map).

The continuous maximal regularity result for the solution of \eqref{eq:PDE_pb} (cf.~\cite[Theorem 3.1, 3.2]{KL23}, \cite[p. 211]{Weis19} and \cite{Pruss01}) is summarized below.
\begin{theorem}[Maximal regularity of continuous PDE, {\cite[Theorem~3.1]{KL23}}]\label{thm:max_reg}
	{\it
		{\it If $u_0 = 0$, then the solution $U$ of
			the pulled-back PDE \eqref{eq:PDE_pb} obeys the following estimate for $p,q\in(1,\infty)$:
			\begin{align}
				\| \partial_t U \|_{L^p(0,T; L^q(\Gamma^0))}
				+
				\|  U \|_{L^p(0,T; W^{2,q}(\Gamma^0))}
				\leq
				C
				\| F \|_{L^p(0,T; L^q(\Gamma^0))}
			\end{align}
			where the constant $C>0$ only depends on $\mathcal G_T$.
		}
	}
\end{theorem}

The main results of this paper are the following two theorems, which concern discrete maximal regularity on discrete stationary surfaces and a discrete evolving surfaces, respectively. 

\begin{theorem}[Maximal regularity of semi-discrete surface FEMs on a stationary surface]
\label{thm:disc_max_reg2}
	{\it On a discrete stationary surface $\Gamma_h(s)$ (for any fixed $s\in [0,T]$), the finite element solution $u_h$ of weak formulation
		\begin{align}\label{eq:scheme_stat}
			\left\{
			\begin{aligned}
				&(\partial_t u_h(t),\varphi_h)_{\Gamma_h(s)} + (\nabla_{\Gamma_h(s)} u_h(t),\nabla_{\Gamma_h(s)} \varphi_h)_{\Gamma_h(s)} 
				=
				(f_h ,  \varphi)_{\Gamma_h(s)}
				&&\forall\, \varphi_h\in S_h(\Gamma_h(s)), \\ 
				&u_h(0)= 0 &&\mbox{on}\,\,\,\Gamma_h(s) ,
			\end{aligned}
			\right.	
		\end{align}
		satisfies the following estimate for any given $p,q\in(1,\infty)$, and $h \leq h_0${\rm:}
		\begin{align}
			\| \partial_t u_h \|_{L^p(\R_+; L^q(\Gamma_h(s)))}
			+
			\| \Delta_{\Gamma_{h}(s),h} u_h \|_{L^p(\R_+; L^q(\Gamma_h(s)))}
			\leq
			C
			\| f_h \|_{L^p(\R_+; L^q(\Gamma_h(s)))}
		\end{align}
		where 
		the constant $C>0$ is independent of $h$ and $s$, but depends on the $\mathcal G_T$ and $T$.
	}
\end{theorem}


\begin{theorem}[Maximal regularity of semi-discrete surface FEM on an evolving surface]
\label{thm:disc_max_reg}
	{\it In the case $u_{h,0} = 0$, the solution $u_h$ of \eqref{eq:scheme} obeys the following estimate for any given $p,q\in(1,\infty)$, and $h \leq h_0${\rm:}
		\begin{align}
			\| \partial_t u_h \|_{L^p_t(0,T; L^q(\Gamma_h(t)))}
			+
			\| \Delta_{\Gamma_h({t}),h} u_h \|_{L^p_t(0,T; L^q(\Gamma_h(t)))}
			\leq
			C
			\| f_h \|_{L^p(0,T; L^q(\Gamma_h(t)))}
		\end{align}
		where $f_h(t,x) = P_h(\Gaht) f^{-\ell}(t,x)$, and the constant $C > 0$ is independent of $h$, but depends on $T$. 
	}
\end{theorem}

The proofs of Theorems \ref{thm:disc_max_reg2}--\ref{thm:disc_max_reg} are presented in the following sections.

\section{Preliminary results on stationary surfaces}
\label{section:stationary surfaces}
\setcounter{equation}{0}

In this section, we develop preliminary results of geometric perturbation estimates, Green’s function estimates and local energy estimates on surfaces $\Gamma(s)$ and $\Gahs$ for some fixed $s\in [0,T]$. For the simplicity of notations, we omit the dependence on $s$ by using $\Gamma$ and $\Gah$ instead to denote the surfaces. The constants in this section will not depend on $s$, but may depend on $T$.

\subsection{Function spaces}

We use the conventional notations of Sobolev spaces $W^{s,q}(\Gamma)$, $s\ge 0$ and $1\leq q\leq\infty$, with abbreviations $L^q:=W^{0,q}(\Gamma)$, $W^{s,q}:=W^{s,q}(\Gamma)$ and $H^s:=W^{s,2}(\Gamma)$. We denote by $H^{-s}(\Gamma)$ the dual space of $H^s_0(\Gamma)$. The latter is defined as the closure of $C^\infty_0(\Gamma)$ in $H^s(\Gamma)$. 

For any given function $f \colon (0,T)\rightarrow W^{s,q}$
we define the following Bochner norm (for space-time functions):  
\begin{align}
	&\|f\|_{L^p(0,T;W^{s,q})} = 
	\big\| \|f(\cdot)\|_{W^{s,q}}\big\|_{L^p(0,T)} ,\quad\forall\,\, 1\leq p,q\leq \infty,\,\, s\in\R   . 
\end{align} 
For any subdomain $D\subset \Gamma$, we define 
\begin{align}\label{Def-HsD}
	\|f\|_{W^{s,q}(D)}:=
	\inf_{\tilde f|_D=f}\|\tilde f\|_{W^{s,q}(\Gamma)} ,\quad\forall\,\, 1\leq q\leq \infty ,\,\, s\in\R ,
\end{align} 
where the infimum extends over all possible
$\tilde f$ defined on $\Gamma$ such that
$\tilde f=f$ in $D$. 
Similarly, for any subdomain $Q\subset {\mathcal Q}=(0,1)\times\Gamma$, 
we define 
\begin{align}\label{DefLpX}
	\|f\|_{L^pW^{s,q}(Q)}:=
	\inf_{\tilde f|_Q=f}\|\tilde f\|_{L^p(0,T;W^{s,q})} 
	,\quad\forall\,\, 1\leq p,q\leq \infty ,\,\,  s\in\R  ,
\end{align} 
where the infimum extends over all possible
$\tilde f$ defined on ${\mathcal Q}$ such that
$\tilde f=f$ in $Q$. 
For more details on these spaces we refer to \cite{AlphonseElliottStinner_2015,AlphonseCaetanoDjurdjevacElliottS_2023}.

In addition, the following notations of inner products will be used:  
\begin{align}\label{inner-products}
	(\phi,\varphi):=\int_\Gamma \phi(x)\varphi(x)\d x,\qquad
	[u,v]:=\int_0^T\int_{\Gamma} u(t,x)v(t,x)\d x \, \d t .
\end{align} 
For any function $w$ defined on ${\mathcal Q}$, We denote $w(t)=w(t,\cdot)$. The notation $1_{0<t<T}$ stands for the characteristic function of time interval $(0,T)$, i.e.,~$1_{0<t<T}(t)=1$ if $t\in(0,T)$ and $1_{0<t<T}(t)=0$ if $t\notin(0,T)$.

\subsection{Properties of the finite element space}
\label{sec:hypo}

For any subdomain $D\subset\Gamma$, we denote by $S_h(D)$ the space of functions in $S_h(\Gamma)$  restricted to the domain $D$, and denote by $S_h^0(D)$ the subspace of $S_h(D)$ consisting of functions which equal zero outside $D$. For $d>0$, we denote by $B_d (D) =\{x\in\Gamma: {\rm dist}(x,D)\leq d\}$ a neighborhood of $D$ in $\Gamma$. On a quasi-uniform triangulation of surface $\Gamma$, there exist positive constants $K $ and $\kappa$ such that the triangulation and the corresponding finite element space $S_h(\Gamma)$ possess the following properties ($K$ and $\kappa$ are independent of the subset $D$ and $h$).
\medskip

\begin{enumerate}[label={\bf (P\arabic*)},ref=\arabic*]\itemsep=5pt

\item {\bf Quasi-uniformity:}

\noindent For all triangles $\tau_l^h$ in the partition,
the diameter $h_l$ of $\tau_l^h$ and the radius $\rho_l$
of its inscribed ball satisfy
$$
K^{-1}h\leq \rho_l \leq h_l\leq Kh .
$$

\item {\bf Inverse inequality:}

\noindent If $D$ is a union of elements in the partition, then, for $0\leq k\leq l\leq 1$ and 
$1\leq q\leq p\leq\infty$,
\begin{align*}
	\|\chi_h\|_{W^{l,p}(D)}
	\leq K h^{-(l-k)-(\dimsurf/q-\dimsurf/p)}\|\chi_h\|_{W^{k,q}(D)} ,
	\quad\forall\,\,\chi_h\in S_h .
\end{align*}

\item {\bf Local approximation and super-approximation:} 

\noindent There exists an operator $I_h \colon H^1(\Gamma) \rightarrow S_h$ with the following properties:

\begin{enumerate}[label=(\arabic*),ref=\arabic*]\itemsep=5pt
	\item 
	For $v\in H^{2}(\Gamma )$ the following estimate holds, for $\Gamma \subset \R^{\dimsurf + 1}$ with $\dimsurf = 1,2,3$, and $h \leq h_0$:
	\begin{align*}
		&\|v-I_hv\|_{L^2} +h \|\nabla(v- I_hv)\|_{L^2} \leq Kh^{2} \|v\|_{H^{2}} 
		.
	\end{align*}

	\item 
	If $d\geq 2h$ then the value of $I_h v$ in $D$ depends only on the value of $v$ in $B_d(D)$. 
	If $d\ge 2h$ and supp$(v)\subset \overline D$, then $I_hv\in S_h^0(B_{d}(D))$.


\item 
If $d\geq 2h$, $\chi=0$ outside $D$
and $|\partial^\beta\chi|\leq Cd^{-|\beta|}$
for all multi-index $\beta$,
then 
%
\begin{align*}
	&\psi_h\in S_h(B_{d}(D))\implies I_h(\chi\psi_h)\in S_h^0(B_{d}(D)),\\
	&\|\chi\psi_h-I_h(\chi\psi_h)\|_{L^2}
	+h\|\chi\psi_h-I_h(\chi\psi_h)\|_{H^1}
	\leq K  h d^{-1}
	\|\psi_h\|_{L^2(B_{d}(D))} .
\end{align*}

\item  If $d\ge 2h$ and $\chi\equiv 1$ on $B_{d}(D)$, then $I_h(\chi\psi_h)=\psi_h$ on $D$. 

\end{enumerate}

\end{enumerate}

Properties (P1)--(P3) hold for any quasi-uniform triangulation with the standard finite element spaces consisting of globally continuous piecewise polynomials of degree $r \geq 1$ (cf.~\cite[Appendix]{Schatz95}). 
\begin{remark}
	Properties (P1)--(P3) hold for parametric surface finite element spaces $S_h(\Gamma(t))$ and $S_h(\Gamma_h(t))$ for all $t\in[0,T]$, see \cite{Dem09,Dziuk2013b}.
\end{remark}

Given $\Gamma$ and $\Gah$,
we define the scalar-valued function $a_h(x)$ and the $\R^{(\dimsurf+1)\times (\dimsurf+1)}$-valued function $B_h(x)$ with $x\in \Gamma(s)$ to be the piecewise smooth prefactors in the following change of variables:
\begin{align}\label{inner-product-1}
	\int_{\Gamma_h} uv &= \int_{\Gamma} a_h(x) u^\ell v^\ell , \\
	\label{inner-product-2}
	\int_{\Gamma_h}\nabla_{\Gamma_h} u\cdot\nabla_{\Gamma_h} v &= \int_{\Gamma} B_h(x) \nabla_{\Gamma} u^\ell\cdot \nabla_{\Gamma} v^\ell .
\end{align}
%
Then the following estimates hold (see \cite{Dem09}):
\begin{lemma}\label{lemma:geo_perturb}
	Let $\Gamma$ and $\Gamma_h$ be as above. Then, for sufficiently small $h \leq h_0$, the prefactors satisfy the geometric estimates:
	\begin{align*}
		\| a_h - 1 \|_{L^\infty(\Gamma)} + \| B_h - I \|_{L^\infty(\Gamma)} \leq C h^{k+1} . \notag
	\end{align*}
\end{lemma} 

Besides, the basic parametric finite element theory says that the $L^p$ and $W^{1,p}$ norms on $\Gamma$ and $\Gah$ are equivalent (cf.~\cite[p. 811]{Dem09}), i.e.~$\| \varphi_h \|_{L^p(\Gamma)} \sim \| \varphi_h \|_{L^p(\Gah)}$ and $\| \nabla_{\Gamma} \varphi_h \|_{L^p(\Gamma)} \sim \| \nabla_{\Gah} \varphi_h \|_{L^p(\Gah)}$ for all $\varphi_h\in S_h$ and $p\in[1,\infty]$. Here the notation $a\sim b$ means there exist a positive constant $C$ such that $C^{-1} a \leq b \leq C a$.

\subsection{Green's functions}
We consider the heat equation on a fixed closed and smooth surface $\Gamma$, i.e.,
\begin{align}\label{eq:HE}
	\left\{\begin{array}{ll}
		\partial_t u(t, x) - \Delta_{\Gamma} u(t, x)=f(t,x)
		&\mbox{in}\,\,\,(0,T)\times \Gamma , \\
		u(0,x)=u_0(x)
		&\mbox{in}\,\,\,\Gamma ,
	\end{array}\right.
\end{align}
and the corresponding semi-discrete finite element scheme:
\begin{align}\label{eq:HE_FEM}
	\left\{\begin{array}{ll}
		(\partial_t u_h(t,x),v_h)_{\Gamma_h}
		+ (\nabla_{\Gamma_h} u_h(t,x),\nabla_{\Gamma_h} v_h)_{\Gamma_h}
		= (f_h(t,x), v_h)_{\Gamma_h} ,
		&\forall \, v_h\in S_h(\Gamma_h), \\[5pt]
		u_h(0,x)=  u_{h,0} ,
	\end{array}\right.
\end{align}
where $f_h = P_h(\Gamma_h)f^{-\ell}$.
According to Lemma \ref{lemma:geo_perturb}, we can equivalently lift the scheme onto $\Gamma$:
\begin{align}\label{eq:HE_FEM_pb}
	\left\{\begin{array}{ll}
		(a_h(x) \partial_t u_h^\ell(t,x),v_h^\ell)_{\Gamma}
		+ (B_h(x) \nabla_{\Gamma} u_h^\ell(t,x),\nabla_{\Gamma} v_h^\ell)_{\Gamma}
		= (a_h(x) f_h^\ell(t,x), v_h^\ell)_{\Gamma} ,
		&\forall \, v_h^\ell\in S_h(\Gamma), \\[5pt]
		u_h^\ell(0,x)=  u_{h,0}^\ell .
	\end{array}\right.
\end{align}
Let $G(t,x,x_0)$ denote the Green's function (i.e.~the heat kernel)
of the parabolic equation \eqref{eq:HE}, i.e.~$G=G(\cdot,\cdot\, ,x_0)$ is the solution
of 
\begin{align}\label{GFdef}
	\left\{\begin{array}{ll}
		\partial_t G(\cdot,\cdot\, ,x_0)-\Delta_\Gamma G(\cdot,\cdot\, ,x_0)=0
		&\mbox{in}\,\,\, (0,T)\times \Gamma,\\
		G(0,\cdot,x_0)= \delta_{x_0}
		&\mbox{in}\,\,\,\Gamma .
	\end{array}\right.
\end{align}
The Li--Yau heat kernel estimate on closed manifold $\Gamma$ with Ricci curvature bounded from below (cf.~\cite[Corollary 3.1]{LY86}) gives
\begin{align}
	G(t,x,y) \leq c_0 t^{-\frac{\dimsurf}{2}} e^{c_1 t} e^{-\frac{d(x,y)^2}{c_0 t}} \qquad t > 0,
\end{align}
where the constant $c_0 > 0$ and $c_1 \geq 0$ only depend on the lower bound of the Ricci curvature. Additionally, $c_1 = 0$ if the Ricci curvature of $\Gamma$ is non-negative everywhere and $c_1 > 0$ for the general case. The self-adjointness of Laplace-Beltrami operator implies the symmetry of the kernel, i.e.~$G(t,x,y) = G(t,y,x)$, and the analyticity of the kernel $G(z,x,y)$ on the right complex plane $z\in\{t+is: t > 0 \}$ (see \cite[Lemma 2]{Davies97}).
Following \cite[Theorem 3.4.8]{Davies89}, the following pointwise bound for the complex time heat kernel $G(z,x,y)$ holds.
\begin{lemma}\label{lemma:guass-est-c}
	The analytic complex time heat kernel on surface $\Gamma$ satisfies 
	\begin{align}\label{eq:guass-est-c}
		| G(z,x,y) | \leq c_0 (\Re (z))^{-\frac{\dimsurf}{2}} e^{c_1 \Re(z)} e^{-\Re \frac{d(x,y)^2}{c_0 z}} ,
	\end{align}
	for all $\Re(z)> 0$ and $x,y \in \Gamma$.
\end{lemma}
Then Cauchy's integral formula says that for all real time $t > 0$
\begin{equation}
	\partial_t^kG(t,x,y)
	=\frac{k!}{2\pi i}
	\int_{|z-t|=\frac{t}{2}}\,\, \frac{G(z,x,y)}{(z-t)^{k+1}}\d z ,
\end{equation}
which, together with \eqref{eq:guass-est-c}, yields the following Gaussian pointwise estimate for the time derivatives of Green's function:
\begin{align}
	&|\partial_t^kG(t,x,x_0)|\leq \frac{C_k}{t^{k+\dimsurf/2}} e^{C_k t} e^{-\frac{d(x-x_0)^2}{C_kt}}, &&
	\forall\, x,x_0\in \Gamma,\,\,\, \forall\, t>0, \,\, k=0,1,2,\dotsc \label{GausEst1} ,
\end{align} 
where the constant $C_k$ depends on $c_0$ and $c_1$ which are related to the Ricci curvature of $\Gamma$.

As constructed in \cite[Lemma 2.2]{Tho00}, for any $x_0\in \bar K^{l}\subseteq\Gamma$ (where $K^{l}$ is a lifted open triangle in the triangulation of $\Gamma$), there exists a function $\tilde\delta_{x_0}\in C_0^\infty(K^{l})$ such that
\begin{align*}
	\chi_h(x_0)=\int_{\Gamma}\chi_h \tilde\delta_{x_0}\d x,
	\quad\forall\,\chi_h\in S_h(\Gamma) ,
\end{align*}
and
\begin{align}
	&\|\tilde\delta_{x_0}\|_{W^{l,p}}
	\leq C h^{-l-\dimsurf(1-1/p)}
	\qquad \mbox{for} \,\,\,1\leq p\leq\infty,
	\,\,\, l=0,1,2,\dotsc , \label{reg-Delta-est}\\
	&\sup_{y\in \Gamma} \int_\Gamma |\tilde\delta_{y}(x)|\d x+
	\sup_{x\in \Gamma}\int_\Gamma |\tilde\delta_{y}(x)|\d y \le C .
	\label{reg-Delta-est-2}
\end{align}

Let $\delta_{x_0}$ denote the Dirac Delta function centered at $x_0\in \Gamma$. In other words, $\int_\Gamma\delta_{x_0}(y)\varphi (y)\d y=\varphi(x_0)$ for arbitrary $\varphi\in C(\Gamma)$.  
Then the discrete Delta function associated to $\Gamma$
$$
\tilde\delta_{h,x_0} :=P_h(\Gamma)  \delta_{x_0}=P_h(\Gamma) \tilde\delta_{x_0} \qquad\forall x_0\in \Gamma
$$ 
decays exponentially away from $x_0$ (cf.~\cite[Lemma 2.3]{Tho00} and \cite[Lemma 6.1]{Tho06}): 
\begin{align}\label{Detal-pointwise}
	& |\tilde\delta_{h,x_0}(x)|=|P_h(\Gamma) \tilde\delta_{x_0}(x)|
	\leq Ch^{-\dimsurf} e^{-\frac{ d(x,x_0) }{K h }} ,
	\quad \forall\, x,x_0\in \Gamma  .
\end{align}

Let $\rg=\rg(\cdot,\cdot\, , x_0)$, $x_0\in \Gamma$ 
be the regularized Green's function
of the parabolic equation on the closed manifold $\Gamma$, defined by 
\begin{align}\label{RGFdef}
	\left\{\begin{array}{ll}
		\partial_t\rg(\cdot,\cdot\, , x_0)-\Delta_\Gamma \rg(\cdot,\cdot\, , x_0)=0
		&\mbox{in}\,\,\,(0,T)\times \Gamma,\\
		\rg(0,\cdot ,x_0)=\tilde\delta_{x_0} 
		&\mbox{in}\,\,\,\Gamma .
	\end{array}\right.
\end{align}
The regularized Green's function can be represented by 
\begin{align}\label{expr-Gamma}
	&\rg(t,x,x_0)=\int_\Gamma  G(t,y,x)\tilde\delta_{x_0}(y)\d y=\int_\Gamma  G(t,x,y)\tilde\delta_{x_0}(y)\d y  . 
\end{align}
From the representation \eqref{expr-Gamma} and \eqref{GausEst1}, one can easily derive that the regularized Green's function $\rg$ also satisfies the Gaussian pointwise estimate in the region $\max( d(x, x_0), \sqrt t) \geq 2h$:
\begin{align}
	&|\partial_t^k\rg(t,x,x_0)|\leq \frac{C_k}{t^{k+\dimsurf/2}} e^{C_k t} e^{-\frac{d(x-x_0)^2}{C_kt}}, \qquad
	\forall\, x,x_0\in \Gamma,\,\,\forall\, t>0,  \label{GausEstGamma}
\end{align} 
with $k=0,1,2,\dotsc$.

Similarly,  for any $x_0\in \bar K\subseteq\Gamma_h$ (where $K$ is an open triangle in the triangulation of $\Gamma_h$),  we denote by $\bar\delta_{x_0}\in C_0^\infty(K)$ the regularized Green's function on $\Gamma_h$ (cf.~\cite[Lemma 2.2]{Tho00}), satisfying the following relations: 
 \begin{align*}
 	\chi_h(x_0)=\int_{\Gamma_h}\chi_h \bar\delta_{x_0}\d x,
 	\quad\forall\,\chi_h\in S_h(\Gamma_h) ,
 \end{align*}
 and
 \begin{align}
 	&\|\bar\delta_{x_0}\|_{W^{l,p}}
 	\leq C h^{-l-\dimsurf(1-1/p)}
 	\qquad \mbox{for}\,\,\,1\leq p\leq\infty,
 	\,\,\, l=0,1,2, \dotsc , \label{reg-Delta-est-3}\\
 	&\sup_{y\in \Gamma_h} \int_{\Gamma_h} |\bar\delta_{y}(x)|\d x+
 	\sup_{x\in \Gamma_h}\int_{\Gamma_h} |\bar\delta_{y}(x)|\d y \le C .
 	\label{reg-Delta-est-4}
 \end{align}
Let $\bar\delta_{h,x_0}$ be the discrete Green's function in $S_h(\Gamma_h)$, defined as 
$$
\bar\delta_{h,x_0} := P_h(\Gah) \delta_{x_0} = P_h(\Gah)  \bar\delta_{x_0} \qquad\forall x_0\in \Gamma_h .
$$  
By the same argument as in a flat domain (see \cite[Lemma 6.1]{Tho06}), it is straightforward to verify that $\bar\delta_{h,x_0}$ exponentially decays away from $x_0$, i.e., 
\begin{align}\label{Detal-pointwise1}
	& |\bar\delta_{h,x_0}(x)|=|P_h(\Gah) \bar\delta_{x_0}(x)|
	\leq Kh^{-\dimsurf} e^{-\frac{ d(x,x_0) }{K h }} ,
	\quad \forall\, x,x_0\in \Gamma_h .
\end{align}
This estimate of the discrete Delta function also implies the boundedness of projection $P_h(\Gah)$ in $L^q(\Gamma_h)$, i.e., 
\begin{align}\label{L^p-stability-P_h} 
	&\|P_h(\Gah) v\|_{L^p(\Gamma_h)} \le C \| v\|_{L^p(\Gamma_h)} \quad\forall\, v\in L^p(\Gamma_h),
	\,\,\,1\le p\le\infty. 
\end{align} 
Let $\rg_h=\rg_h(\cdot,\cdot,  x_0)\in S_h(\Gamma_h), x_0\in \Gah,$ be the discrete Green's function associated to \eqref{eq:HE_FEM}, i.e., 
\begin{align}
	\left\{\begin{array}{ll}
		(\partial_t \rg_h(t,\cdot,x_0),v_h)_{\Gamma_h}
		+ (\nabla_{\Gamma_h} \rg_h(t,\cdot,x_0),\nabla_{\Gamma_h} v_h)_{\Gamma_h}
		= 0 ,
		&\forall \, v_h\in S_h(\Gamma_h), \\[5pt]
		\rg_h(0,\cdot,x_0) =  \bar\delta_{h,x_0} .
	\end{array}\right.
\end{align}
Then the discrete Green's function $\rg_h(t, x, x_0)$ is symmetric with respect to $x$ and $x_0$. Moreover, in view of the relations in \eqref{inner-product-1}--\eqref{inner-product-2}, the lifted discrete Green's function $\rg_h^\ell(t,x,x_0)=\rg_h(t,q^{-1}(x),  q^{-1}(x_0))$, for $x,x_0\in\Gamma$, satisfies the following weak formulation on $\Gamma$: 
\begin{align}\label{EqGammh}
	\left\{\begin{array}{ll}
		(a_h(\cdot) \partial_t\rg_h^\ell(t,\cdot,x_0),v_h^\ell)_{\Gamma}
		+ (B_h(\cdot) \nabla_{\Gamma} \rg_h^\ell(t,\cdot,x_0),\nabla_{\Gamma} v_h^\ell)_{\Gamma}
		=0 ,
		&\forall \, v_h\in S_h(\Gamma), \\[5pt]
		\rg_h^\ell(0,\cdot,x_0)=  \bar\delta_{h,q^{-1}(x_0)} ^\ell .
	\end{array}\right.
\end{align}

By testing \eqref{eq:HE} and \eqref{eq:HE_FEM} with the backward Green's function and discrete Green's functions, respectively, the solutions of \eqref{eq:HE} and \eqref{eq:HE_FEM} can be represented by 
\begin{align}
	u(t,x_0)&= 
	\int_\Gamma G(t,x,x_0)u_0(x)\d x
	+ \int_0^t\int_\Gamma G(t-s,x,x_0)f(s,x)\d x\d s ,\label{eq:conv_u}\\
	u_h(t,x_0)
	&=
	\int_{\Gamma_h}  \rg_h(t,x,x_0)u_{h,0}(x)\d x
	+\int_0^t\int_{\Gamma_h}  \rg_h(t-s,x,x_0)f_h(s,x)\d x\d s \notag\\
	&=
	\int_{\Gamma} a_h(x) \rg_h^\ell (t,x,q(x_0))u_{h,0}^\ell (x)\d x
	+\int_0^t\int_{\Gamma} a_h(x) \rg_h^\ell (t-s,x,q(x_0))f_h^\ell(s,x)\d x\d s , \label{eq:conv_uh}
\end{align}
where the last equality uses \eqref{inner-product-1} and relation $f_h = P_h(\Gamma_h)f^{-\ell}$. 

\begin{remark}
The construction of the regularized Green's function is not unique, since we can indeed choose any smooth weight function in the reference element in the proof of \cite[Lemma 2.2]{Tho00}. 
By definition, $\bar\delta_{h,x_0}$ and $\tilde\delta_{h,x_0}$ are the $L^2$ projections of the Dirac delta functions onto the closed subspaces $S_h(\Gamma_h)$ and $S_h(\Gamma)$ (in the sense of distribution), respectively, and therefore they do not depend on the definitions of the regularized Green's functions $\bar\delta_{x_0}$ and $\tilde\delta_{x_0}$. 

\end{remark}
In order to see the relation between the regularized Delta functions $\bar\delta_{x_0}$ on $\Gamma_h$ and $\tilde\delta_{x_0}$ on $\Gamma$, respectively, we can consider a fixed point $x_0\in \bar K\subset\Gamma_h$ and push the regularized delta function $\bar\delta_{x_0}$ to $\Gamma$ through the distance projection $q$ (which projects any point in a neighborhood of $\Gamma$ onto $\Gamma$). 
This leads to the following relations: 
\begin{align}
	\bar\delta_{x_0} = (a_h^{-1})^{-\ell}\tilde\delta_{q(x_0)}^{-\ell} 
	\qquad\mbox{and}\qquad 
	\tilde\delta_{q(x_0)} = a_h \bar\delta_{x_0}^\ell 
	\qquad\mbox{for all}\,\,\, x_0\in\Gamma_h . 
\end{align}
This relation can be used to show the following estimate for the discrete Green's functions. 

\begin{lemma}\label{lemma:delta}
	Let $h \leq h_0$, then the discrete Green's functions satisfy the estimate
	\begin{align}
		\| \bar\delta_{h,x_0}^\ell - \tilde\delta_{h,q(x_0)} \|_{L^p(\Gamma)} 
		\leq
		C h^{k+1} \min \{\|  \tilde\delta_{q(x_0)} \|_{L^p(\Gamma)}, \|  \bar\delta_{x_0} \|_{L^p(\Gamma_h)} \} ,
	\end{align}
	for all $x_0\in\Gamma_h$ and $p\in [1,\infty]$.
\end{lemma}
\begin{proof}
	By the construction of the regularized delta function on $\Gamma$ and $\Gamma_h$, 
	\begin{align}
		&\| \bar\delta_{h,x_0}^\ell - \tilde\delta_{h,q(x_0)} \|_{L^p(\Gamma)} 
		= \| (P_h(\Gah) \bar\delta_{x_0})^\ell - P_h(\Gamma)\tilde\delta_{q(x_0)} \|_{L^p(\Gamma)} \notag\\
		&= \| (P_h(\Gah) \bar\delta_{x_0})^\ell - P_h(\Gamma) (a_h \bar\delta_{x_0}^\ell) \|_{L^p(\Gamma)} 
		 \notag\\
		&\leq \| (P_h(\Gah) \bar\delta_{x_0})^\ell - P_h(\Gamma) \bar\delta_{x_0}^\ell \|_{L^p(\Gamma)} 
		+ \| P_h(\Gamma) (1 - a_h)\bar\delta_{x_0}^\ell \|_{L^p(\Gamma)} 
		 \notag\\
		&\leq \sup_{v_h\in S_h(\Gamma), \| v_h \|_{L^{p^\prime}(\Gamma)} = 1}( (P_h(\Gah) \bar\delta_{x_0})^\ell - P_h(\Gamma) \bar\delta_{x_0}^\ell, v_h)_\Gamma
		+ C \| 1 - a_h \|_{L^\infty(\Gamma)} \| \bar\delta_{x_0}^\ell \|_{L^p(\Gamma)} 
		\notag\\
		&= \sup_{v_h\in S_h(\Gamma), \| v_h \|_{L^{p^\prime}(\Gamma)} = 1}(\bar\delta_{x_0}, P_h(\Gah) (a_h^{-1})^{-\ell} v_h - (a_h^{-1})^{-\ell} v_h)_{\Gamma_h}
		+ C \| 1 - a_h \|_{L^\infty(\Gamma)} \| \bar\delta_{x_0}^\ell \|_{L^p(\Gamma)} 
		\notag\\
		&\leq \sup_{v_h\in S_h(\Gamma), \| v_h \|_{L^{p^\prime}(\Gamma)} = 1}(\bar\delta_{x_0}, P_h(\Gah) (a_h^{-1})^{-\ell} v_h - v_h)_{\Gamma_h} \notag\\
		&\quad+ \sup_{v_h\in S_h(\Gamma), \| v_h \|_{L^{p^\prime}(\Gamma)} = 1}(\bar\delta_{x_0}, v_h - (a_h^{-1})^{-\ell} v_h)_{\Gamma_h} 
		+ C \| 1 - a_h \|_{L^\infty(\Gamma)} \| \bar\delta_{x_0}^\ell \|_{L^p(\Gamma)} 
		\notag\\
		&\leq \sup_{v_h\in S_h(\Gamma), \| v_h \|_{L^{p^\prime}(\Gamma)} = 1} \| P_h(\Gah) ((a_h^{-1})^{-\ell} - 1) v_h \|_{L^{p^\prime}(\Gah)} \| \bar\delta_{x_0} \|_{L^p(\Gah)}  \notag\\
		&\quad+ \sup_{v_h\in S_h(\Gamma), \| v_h \|_{L^{p^\prime}(\Gamma)} = 1} \| ((a_h^{-1})^{-\ell} - 1) v_h \|_{L^{p^\prime}(\Gah)} \| \bar\delta_{x_0} \|_{L^p(\Gah)} 
		+ C h^{k+1} \| \bar\delta_{x_0}^\ell \|_{L^p(\Gamma)} 
		\notag\\
		&\leq C h^{k+1} \| \bar\delta_{x_0} \|_{L^p(\Gah)} , \notag
	\end{align}
	where we have used the $L^\infty$ estimate of $1 - a_h$ (Lemma \ref{lemma:geo_perturb}) and the $L^p$ stability of $P_h(\Gamma_h)$ shown in \eqref{L^p-stability-P_h}.
	Since $\bar\delta_{x_0}^\ell= a_h^{-1} \tilde\delta_{q(x_0)}$, we know that $\| \bar\delta_{x_0} \|_{L^p(\Gah)} $ and $\| \tilde\delta_{q(x_0)} \|_{L^p(\Gamma)} $ are equivalent.
\end{proof}

\subsection{$W^{1,q}$-stability of Ritz projection}
\label{sec:Rh}

We consider the Ritz projection $R_h\colon H^1(\Gamma)\rightarrow S_h(\Gamma)$ associated to the operator $-\Delta_{\Gamma} + 1$ on $\Gamma$, defined as 
\begin{align}
	(\nabla_{\Gamma} (u-R_hu), \nabla_{\Gamma} \varphi_h^\ell)_{\Gamma}
	+
	(u-R_hu, \varphi_h^\ell)_{\Gamma}
	= 0 
\quad\forall \varphi_h^\ell \in S_h(\Gamma) . \notag
\end{align}
From \cite[Theorem 3.2]{Dem09} we know that the Ritz projection defined above is $W^{1,\infty}$ stable.
Consequently the $W^{1,q}$-stability of $R_h$ for $q\in [1,\infty]$ follows from interpolation  (cf.~\cite[Theorem~3.70]{Aubin82}) and duality via Hodge decomposition (cf.~\cite[Eq.~(10.101)]{Gia13}).

\subsection{Dyadic decomposition of the domain ${\mathcal Q}=(0,1)\times\Gamma$}
\label{SecGF} 
In the proof of Theorem \ref{thm:disc_max_reg2}, we decompose the domain ${\mathcal Q}=\Gamma_T$ with $T=1$, i.e.~${\mathcal Q}=(0,1)\times\Gamma$, into subdomains, and present estimates of the finite element solutions in each subdomain. The following dyadic decomposition of ${\mathcal Q}$ was introduced in \cite{Schatz98}. The readers may skip this subsection if they are familiar with such dyadic decompositions.

For any integer $j$, we define $d_j=2^{-j}$.  For a given $x_0\in\Gamma$, 
we let $J_1=1$, $J_0=0$ and $J_*$ be an integer satisfying $2^{-J_*}= C_*h$ 
with $C_*\geq 16$ to be determined later.  
If 
\begin{align}\label{h-condition}
	h<1/(4C_*)  ,
\end{align}
then 
\begin{align}
	2\leq J_*=\log_2[1/(C_*h)]\leq \log_2(2+1/h) .
\end{align} 
Let  
\begin{align*}
	&Q_*(x_0)=\{(t,x)\in\Gamma_T: \max (d(x,x_0),t^{1/2})\leq d_{J_*}\}, 
	\\
	&\Gamma_*(x_0)=\{x\in \Gamma: d(x,x_0) \leq d_{J_*}\} \, . 
\end{align*} 
We define 
\begin{align*} 
	&Q_j(x_0)=\{(t,x)\in \Gamma_T:d_j\leq
	\max (d(x,x_0),t^{1/2})\leq2d_j\}  &&\mbox{for}\,\,\, j\ge 1,
	\\
	&\Gamma_j(x_0)=\{x\in \Gamma: d_j\leq d(x,x_0) \leq2d_j\}   
	&&\mbox{for}\,\,\, j\ge 1,
	\\
	&D_j(x_0)=\{x\in \Gamma:  d(x,x_0) \leq2d_j\}  
	&&\mbox{for}\,\,\, j\ge 1,
\end{align*} 
and 
\begin{align*} 
	&Q_0(x_0)= 
	{\mathcal Q}\big\backslash\big( \cup_{j=1}^{J_*}Q_{j}(x_0)\cup Q_*(x_0)\big) , 
	\\
	&\Gamma_0(x_0) 
	=\Gamma\big\backslash\big( \cup_{j=1}^{J_*}\Gamma_{j}(x_0)\cup \Gamma_*(x_0)\big).
\end{align*}
For $j<0$, we simply define
$Q_{j}(x_0)=\Gamma_{j}(x_0)=\emptyset$.
For all integer $j\ge 0$, we define
\begin{align*}
	&\Gamma_j'(x_0)=\Gamma_{j-1}(x_0)\cup\Gamma_{j}(x_0)\cup\Gamma_{j+1}(x_0),
	\quad Q_j'(x_0)=Q_{j-1}(x_0)\cup Q_{j}(x_0)\cup Q_{j+1}(x_0), \\
	&\Gamma_j''(x_0)=\Gamma_{j-2}(x_0)\cup\Gamma_{j}'(x_0)\cup\Gamma_{j+2}(x_0),
	\quad
	Q_j''(x_0)=Q_{j-2}(x_0)\cup Q_{j}'(x_0)\cup Q_{j+2}(x_0),\\
	&D_j'(x_0)=D_{j-1}(x_0)\cup D_{j}(x_0) ,
	\quad D_j''(x_0)=D_{j-2}(x_0)\cup D_{j}'(x_0) .
\end{align*}
Then we have
\begin{align}\label{decomposition}
	&\Gamma_T=\bigcup^{J_*}_{j=0}Q_j(x_0)\,\cup Q_*(x_0)
	\quad\mbox{and}\quad
	\Gamma=\bigcup^{J_*}_{j=0}\Gamma_j(x_0)\,\cup \Gamma_*(x_0) .
\end{align}
We refer to $Q_*(x_0)$ as the ``innermost" set.
We shall write $\sum_{*,j}$ when the innermost set is included and
$\sum_j$ when it is not. When $x_0$ is fixed, if there is no ambiguity, 
we simply write
$Q_j=Q_j(x_0)$, $Q_j'=Q_j'(x_0)$, $Q_j''=Q_j''(x_0)$, 
$\Gamma_j=\Gamma_j(x_0)$, $\Gamma_j'=\Gamma_j'(x_0)$ and 
$\Gamma_j''=\Gamma_j''(x_0)$.

We shall use the notations
\begin{align}\label{Q-norm}
	&\|v\|_{k,D}=\biggl(\int_D\sum_{|\alpha|\leq
		k}|\partial_{\Gamma}^\alpha v|^2\d x\biggl)^{\frac{1}{2}},\qquad
	\vertiii{v}_{k,Q}=\biggl(\int_Q\sum_{|\alpha|\leq
		k}|\partial_{\Gamma}^\alpha v|^2\d x\d t\biggl)^{\frac{1}{2}} ,
\end{align}
for any subdomains $D\subset\Gamma$ and $Q\subset (0,1)\times\Gamma $.
Throughout this paper, we denote by $C$ a generic positive constant that 
is independent of $h$, $x_0$ and $C_*$ (until $C_*$ is determined 
in Section \ref{section:maximal regularity on stationary surfaces}). To simplify the notations, we also denote $d_*=d_{J_*}$.

\subsection{Local energy estimates}
\label{sec:local}

Since the Gagliardo–Nirenberg interpolation inequality still holds on closed manifold (cf.~\cite[Theorem 3.70]{Aubin82}), the standard energy estimates of heat equation (cf.~\cite[Lemma 2.1 of Chapter III]{Lady88}) are also valid here. Consequently we have the following local estimates of (regularized) Green's function (cf.~\cite[Lemma 4.1]{Li-MCOM-2019} and \cite[Eq. (4.9)]{Li-MCOM-2019} with $\alpha=1$).
\begin{lemma}\label{GFEst1}
	{\it The Green's function $G$ defined in \eqref{GFdef} and the regularized Green's function $H$ defined in \eqref{RGFdef} satisfy the following estimates:
		\begin{align}
			&d_j^{-5+\dimsurf/2}\|\rg(\cdot,\cdot,x_0)\|_{L^\infty(Q_j(x_0))}
			+d_j^{-5}\vertiii{\nabla_{\Gamma} \rg(\cdot,\cdot,x_0)}_{L^2(Q_j(x_0))}
			\nn\\ 
			&
			+d_j^{-4}\vertiii{ \rg(\cdot,\cdot,x_0)}_{L^2H^{2}(Q_j(x_0))}  +d_j^{-2}\vertiii{ \partial_{t} 
				\rg(\cdot,\cdot,x_0)}_{L^2H^{2}(Q_j(x_0))} \nn\\
			& 
			+\vertiii{ \partial_{tt}\rg(\cdot,\cdot,x_0)}_{L^2H^{2}(Q_j(x_0))}
			\leq Cd_j^{-\dimsurf/2-5}
			, \label{GFest01}\\[10pt]
			& 
			\vertiii{\partial_t\rg(\cdot,\cdot,x_0)}_{L^2(Q_j(x_0))}
			+
			d_j\vertiii{\nabla_{\Gamma} \partial_t\rg(\cdot,\cdot,x_0)}_{L^2(Q_j(x_0))}
			+
			d_j^{2}\vertiii{\partial_{tt}\rg(\cdot,\cdot,x_0)}_{L^2(Q_j(x_0))}
			 \notag\\
			&\quad
			+
			d_j^{3}\vertiii{\nabla_{\Gamma} \partial_{tt}\rg(\cdot,\cdot,x_0)}_{L^2(Q_j(x_0))}
			\leq Cd_j^{-\dimsurf/2-1}
			, \label{GFest02}\\[10pt]
			&\|G(\cdot,\cdot,x_0) \|_{L^{\infty}H^{2}(\cup_{k\leq
					j}Q_k(x_0))}
			+d_j^2\|\partial_tG(\cdot,\cdot,x_0) \|_{L^{\infty}H^{2}(\cup_{k\leq
					j}Q_k(x_0))}\leq Cd_j^{-\dimsurf/2-2} \label{GFest03} .
		\end{align}
%
	}
\end{lemma}
\begin{remark}\label{rmk:high_est}
	The spatial high-order counterpart of \eqref{GFest03}, i.e.
	\begin{align}
		\|G(\cdot,\cdot,x_0) \|_{L^{\infty}H^{k}(\cup_{k\leq
				j}Q_k(x_0))}
		\leq Cd_j^{-\dimsurf/2-k} , \quad k\in \N, \notag
	\end{align}
	 follows from a standard induction argument by differentiating the localized equation in the spatial direction. Note that on the surface the surface derivatives do not commute with each other (cf. \cite[Lemma 5.1, Item 4]{BL22B}). However the commutator is of lower order and therefore dose not affect the dominant stability in the energy estimates.
\end{remark}

We also need the following local energy estimate for the error equation, which can be proved in a similar way as \cite[Lemma 5.1]{Li-MCOM-2019}. The proof of Lemma \ref{LocEEst} can be found in Appendix.
\begin{lemma}\label{LocEEst} 
	{\it 
		Suppose that $\phi\in L^2(0,T;H^1_0(\Gamma))\cap H^1(0,T;L^2(\Gamma))$ and 
		$\phi_h^\ell\in H^1(0,T; S_h(\Gamma))$ satisfy the equation 
		\begin{align}\label{eq:err_eq_1}
			(\partial_t\phi-a_h\partial_t\phi_h^\ell),\chi_h^\ell)_\Gamma
			+(\nabla_\Gamma \phi - B_h\nabla_\Gamma \phi_h^\ell,\nabla_{\Gamma} \chi_h^\ell)_\Gamma =0, 
			\quad\forall\, \chi_h^\ell\in  S_h(\Gamma) , \,\, \mbox{a.e.}\,\, t>0, 
		\end{align} 
		with $\phi(0)=0$ in $\Gamma_j''$. 
		Then 
		\begin{align} 
			\hspace{-8pt}
			&\vertiii{\partial_t(\phi-\phi_h^\ell)}_{Q_j} 
			+ d_j^{-1}\vertiii{\phi-\phi_h^\ell}_{1,Q_j} \notag\\ 
			& \leq 
			C\epsilon^{-3}\big(I_j(\phi_h^\ell(0))+X_j(I_h\phi-\phi) 
			+ Y_j(\phi) 
			+d_j^{-2} \vertiii{\phi-\phi_h^\ell}_{Q_j'}\big) \nn\\
			&\quad 
			+ (C h^{1/2}d_j^{-1/2} + C\epsilon^{-1}hd_j^{-1} + \epsilon) \big(
			\vertiii{\partial_t(\phi-\phi_h^\ell)}_{Q_j'}
			+d_j^{-1}\vertiii{\phi-\phi_h^\ell}_{1,Q_j'}\big), 
			\label{LocEngErr} 
		\end{align} 
		where
		\begin{align*}
			&I_j(\phi_{h}^\ell(0))=\|\phi_h^\ell(0)\|_{1,\Gamma_j'} + d_j^{-1}\|\phi_h^\ell(0)\|_{\Gamma_j'} \, ,\\[5pt]
			&X_j(I_h\phi-\phi)=
			d_j\vertiii{\partial_t(I_h\phi-\phi)}_{1,Q_j'}
			+\vertiii{\partial_t(I_h\phi-\phi)}_{Q_j'} 
			\nn\\
			&\qquad\qquad\qquad\,\, 
			+d_j^{-1}\vertiii{I_h\phi-\phi}_{1,Q_j'}+
			d_j^{-2}\vertiii{I_h\phi-\phi}_{Q_j'}  \, , \\
			& Y_j(\phi) = 
				h^{k+1}
				\bigg(d_j \vertiii{\nabla_\Gamma \partial_t \phi}_{Q_j'}
				+  \vertiii{\partial_t \phi}_{Q_j'}
				+ d_j^{-1} \vertiii{\nabla_\Gamma \phi}_{Q_j'}
				+ d_j^{-2} \vertiii{\phi}_{Q_j'}\bigg)
			 .
		\end{align*}
		The positive constant $C$ is independent of $h$, $j$ and $C_*$; the norms $\vertiii{\cdot}_{k,Q_j'}$ and $\vertiii{\cdot}_{k,\Gamma_j'}$ are defined in \eqref{Q-norm}. 
		
	}
\end{lemma}

\begin{remark}
The term $Y_j(\phi)$ is generated by approximating $\Gamma$ with $\Gamma_h$. 
\end{remark}

\section{Discrete maximal regularity on a stationary surface\\
(Proof of Theorem \ref{thm:disc_max_reg2})}
\label{section:maximal regularity on stationary surfaces}
\setcounter{equation}{0}

We shall prove the following key lemma using the local energy estimates in Section \ref{sec:local}.
\begin{lemma}\label{LemGm2}
	{\it
		The functions $\rg_h(t,x,x_0)$, $\rg(t,x,x_0)$ and 
		$F(t,x,x_0):=\rg_h^\ell(t,x,x_0)-\rg(t,x,x_0)$ satisfy 
		\begin{align}
			&\sup_{x_0\in\Gamma} \|\partial_tF(\cdot,\cdot ,x_0)\|_{L^1((0,+\infty)\times\Gamma)}  
			+ \sup_{x_0\in\Gamma} \|t\partial_{tt}F(\cdot,\cdot , x_0)\|_{L^1((0,+\infty)\times\Gamma)} \leq C ,
			\label{L1Ft}\\
			&
			\sup_{x_0\in\Gamma_h} \|\partial_t\rg_h(t,\cdot, x_0)\|_{L^1(\Gah)}\leq Ce^{-\lambda_0t} , \qquad\forall\, t\ge 1, \label{L1Gammatx0}
		\end{align} 
		where the constants $C>0$ and $\lambda_0 \geq 0$ are independent of $h$. 
	}
\end{lemma}

To simplify the notation, in this subsection we continue to omit the dependence of surface $\Gamma(s)$ on $s\in[0,T]$. Additionally, we relax the dependence of the Green's functions on $x_0\in\Gamma$ by denoting  
$$
\rg_h^\ell(t)=\rg_h^\ell(t,\cdot,  x_0),\qquad
\rg(t)=\rg(t,\cdot,  x_0)\quad\mbox{and}\quad
F(t)=\rg_h^\ell(t) -\rg(t) . 
$$ 
Since the coefficient $a_h$ and $B_h$ are bounded everywhere, the same proofs in \cite[Section 4.3--4.4]{Li-MCOM-2019} using Lemma \ref{LemGm2} can be carried over almost verbatim here. 
For example, by differentiating \eqref{eq:conv_uh} in time, we obtain
\begin{align}
	\partial_t u_h(t,q^{-1}(x_0)) 
	&=
	\int_0^t \int_\Gamma a_h(x) \partial_t F(t-s,x,x_0) f_h^\ell(s,x) \d x \d s 
	\notag\\
	&\quad+
	\int_0^t \int_\Gamma a_h(x) \partial_t \rg(t-s,x,x_0) f_h^\ell(s,x) \d x \d s 
	+
	a_h(x_0) f_h^\ell(t,x_0)
	\notag\\
	&=: \mathcal M_h (a_h f_h^\ell) + \mathcal K_h (a_h f_h^\ell) + a_h f_h^\ell ,\quad x_0\in \Gamma \notag,
\end{align}
where $\mathcal M_h$ and $\mathcal K_h$ are the linear operators whose kernel are $\partial_t F$ and $\partial_t \rg$ respectively. Therefore, by using Lemma \ref{LemGm2} and the arguments in \cite[Section 4.3--4.4]{Li-MCOM-2019}, we can obtain the following result for $p,q\in [1,\infty)$:
\begin{align}
	\| \partial_t u_h^\ell \|_{L^p(\R_+; L^q(\Gamma))}
	\leq
	C \| a_h f_h^\ell \|_{L^p(\R_+; L^q(\Gamma))}
	\leq
	C \| f_h^\ell \|_{L^p(\R_+; L^q(\Gamma))} . \notag
\end{align}
This proves the maximal regularity result in Theorem \ref{thm:disc_max_reg2} (by choosing $\Gamma=\Gamma(s)$ and $\Gah=\Gahs$ for any fixed $s\in [0,T]$).
\begin{remark}
	Analogous to \cite[Section 4.2]{Li-MCOM-2019}, we also have the analyticity and maximum-norm stability (or discrete maximum principle, for the continuous maximum principle on closed manifold $\Gamma$; see \cite[Theorem 15.2]{Li12}) results for the discrete heat semigroup on $\Gah$:
	The semi-discrete finite element solution $u_h=E_h(t)u_{h,0}$ of equation
	\begin{align}\notag
		\left\{
		\begin{aligned}
			&(\partial_t u_h(t),\varphi_h)_{\Gamma_h} + (\nabla_{\Gamma_h} u_h(t),\nabla_{\Gamma_h} \varphi_h)_{\Gamma_h} 
			=
			0
			&&\forall\, \varphi_h\in S_h,  \\ 
			&u_h(0)= u_{h,0} &&\mbox{on}\,\,\,\Gamma  ,
		\end{aligned}
		\right.	
	\end{align}
	satisfies the following estimates for $q\in[1,\infty)$: 
	\begin{align}
		&\sup_{t> 0} (\| E_h(t) u_{h,0} \|_{L^q(\Gah)} + t \| \partial_t E_h(t) u_{h,0} \|_{L^q(\Gah)} )
		\leq
		C \| u_{h,0} \|_{L^q(\Gah)} , \notag\\
		&\sup_{t> 0} \| E_h(t) u_{h,0} \|_{L^\infty(\Gah)}
		\leq
		C \| u_{h,0} \|_{L^\infty(\Gah)} . \notag\\
	\end{align}
\end{remark}



In order to prove Lemma \ref{LemGm2}, we apply the local energy estimate in Lemma \ref{LocEEst} to estimate $\| \partial_t F\|_{L^1((0,+\infty)\times \Gamma)} + \| t\partial_{tt} F\|_{L^1((0,+\infty)\times \Gamma)} $. The estimation consists of two parts: The first part concerns estimates for $t \in (0, 1)$, and the second part concerns estimates for $t\ge 1$, which is a consequence of the smoothing property of parabolic equations. 

\medskip

{\it Part I.}$\,\,\,$ 
First, we present estimates in the domain ${\mathcal Q}=(0,1)\times\Gamma$ with the restriction $h<1/(4C_*) $; see \eqref{h-condition}. In this case, the basic energy estimate gives 
\begin{align} 
&\|\partial_t\rg\|_{L^2({\mathcal Q})}
+\|\partial_t\rg_h^\ell\|_{L^2({\mathcal Q})}
\leq C(\|\rg(0)\|_{H^1} +\|\rg_h^\ell(0)\|_{H^1})
\leq Ch^{-1-\dimsurf/2} ,  \label{EstGamma-Q1} \\
&\|\rg\|_{L^\infty L^2({\mathcal Q})} 
+\|\rg_h^\ell\|_{L^\infty L^2({\mathcal Q})} 
\leq C(\|\rg(0)\|_{L^2} +\|\rg_h^\ell(0)\|_{L^2})
\leq Ch^{-\dimsurf/2} ,\\
&\|\nabla_{\Gamma} \rg\|_{L^2({\mathcal Q})} +\|\nabla_{\Gamma} \rg_h^\ell\|_{L^2({\mathcal Q})}
\leq C(\|\rg(0)\|_{L^2} +\|\rg_h^\ell(0)\|_{L^2})
\leq Ch^{-\dimsurf/2} , \label{EstGamma-Q3} \\
&\|\partial_{tt}\rg\|_{L^2({\mathcal Q})}
+\|\partial_{tt}\rg_h^\ell\|_{L^2({\mathcal Q})}
\leq C(\|\Delta_{\Gamma}\rg(0)\|_{H^1} +\|\Delta_{\Gamma,h}\rg_h^\ell(0)\|_{H^1})
\leq Ch^{-3-\dimsurf/2} ,  \label{EstGamma-Q4}  \\
&\|\nabla_{\Gamma} \partial_t\rg\|_{L^2({\mathcal Q})}
+\|\nabla_{\Gamma} \partial_t\rg_h^\ell\|_{L^2({\mathcal Q})}
\leq C(\|\Delta_{\Gamma}\rg(0)\|_{L^2} +\|\Delta_{\Gamma,h}\rg_h^\ell(0)\|_{L^2})
\leq Ch^{-2-\dimsurf/2} , \label{EstGamma-Q5}
\end{align} 
where we have used \eqref{reg-Delta-est} and \eqref{Detal-pointwise} to estimate $\rg(0)$ and $\rg_h(0)$, respectively. Hence, we have 
\begin{align}\label{L2Gamma-Q}
\vertiii{\rg}_{Q_*} +\ll\rg_h^\ell\ll_{Q_*} 
\le Cd_*\|\rg\|_{L^\infty L^2(Q_*)}+Cd_*\|\rg_h^\ell\|_{L^\infty L^2(Q_*)}
\le Cd_* h^{-\dimsurf/2} \le CC_*h^{1-\dimsurf/2} .
\end{align} 
Since the volume of $Q_j$ is $Cd_j^{2+\dimsurf}$, we can decompose 
$ \|\partial_tF\|_{L^1(\mathcal Q)}+\|t\partial_{tt}F\|_{L^1(\mathcal Q)}$ in the following way:
\begin{align}\label{Bd31K2}
&\|\partial_tF\|_{L^1(\mathcal Q)}+\|t\partial_{tt}F\|_{L^1(\mathcal Q)}  \nn\\
&\leq  \|\partial_{t} F\|_{L^1(Q_*)}
+ \|t\partial_{tt} F\|_{L^1(Q_*)} 
+\sum_{j}\big(\|\partial_{t} F\|_{L^1(Q_j)}
+ \|t\partial_{tt} F\|_{L^1(Q_j)}
\big)  \nn\\
&\leq Cd_{*}^{1+\dimsurf/2} \big(\vertiii{\partial_tF}_{Q_*} 
+ d_*^2\vertiii{\partial_{tt} F}_{Q_*}
\big) 
+\sum_{j}Cd_{j}^{1+\dimsurf/2}
\big(\vertiii{\partial_tF}_{Q_j} 
+d_j^2\vertiii{\partial_{tt} F}_{Q_j}
\big) \nn\\
&\leq CC_*^{3+\dimsurf/2}
+ {\mathscr K} ,
\end{align} 
where we have used \eqref{EstGamma-Q1}, \eqref{EstGamma-Q3} and \eqref{EstGamma-Q4} to estimate 
$$
Cd_{*}^{1+\dimsurf/2} \big(\vertiii{\partial_tF}_{Q_*} 
+d_*^2\vertiii{\partial_{tt} F}_{Q_*} \big)  , 
$$ 
and introduced the following notation: 
\begin{align}\label{express-K}
{\mathscr K}  :&=
\sum_{j} d_{j}^{1+\dimsurf/2}(d_j^{-1}\vertiii{F}_{1,Q_j}
+\vertiii{\partial_tF}_{Q_j}
+d_{j}\vertiii{\partial_tF}_{1,Q_j}+d_{j}^2\vertiii{\partial_{tt}F}_{Q_j})  . 
\end{align} 

It remains to estimate ${\mathscr K} $. To this end, in view of \eqref{RGFdef} and \eqref{EqGammh}, 
we set ``$\phi_h=H_h$, $\phi=H$, $\phi_h(0)= P_h(\Gah) \bar\delta_{q^{-1}(x_0)}$ and $\phi(0)=\tilde\delta_{x_0}$'' 
and ``$\phi_h=\partial_t H_h$, $\phi=\partial_t H$, $\phi_h^\ell (0)=\Delta_{\Gamma,h} (P_h(\Gah) \bar\delta_{q^{-1}(x_0)})^\ell$ and $\phi(0)=\Delta_{\Gamma}\tilde\delta_{x_0}$''
in Lemma \ref{LocEEst}, respectively.
Then we obtain 
\begin{align}\label{F1tE}
d_j^{-1}\vertiii{F}_{1,Q_j}+\vertiii{\partial_tF}_{Q_j} 
&\le 
C\epsilon^{-3}(\widehat{I_j}+\widehat{X_j} + \widehat{Y_j} + d^{-2}_j\vertiii{ F }_{Q'_j} )  \\
&\quad 
+ (C h^{1/2}d_j^{-1/2} + C\epsilon^{-1}hd_j^{-1} + \epsilon) \big( d_j^{-1}\vertiii{F}_{1,Q_j'}
+ \vertiii{\partial_{t}F }_{Q_j'}\big) \qquad \quad\,\, \nn
\end{align}
and
\begin{align}\label{F1ttE}
d_{j}\vertiii{\partial_tF}_{1,Q_j}+d_{j}^2\vertiii{\partial_{tt}F}_{Q_j} 
&\le 
C \epsilon^{-3} (\overline {I_j}+\overline{X_j} + \overline{Y_j} + \vertiii{ \partial_tF }_{Q'_j} )  \\
&\quad 
+ (C h^{1/2}d_j^{-1/2} + C\epsilon^{-1}hd_j^{-1} + \epsilon) \big(  d_j\vertiii{\partial_{t}F}_{1,Q_j'}
+d_{j}^2\vertiii{\partial_{tt}F }_{Q_j'} \big) 
\, , \nn
\end{align}
respectively. 
By using interpolation error estimate, \eqref{Detal-pointwise1} (exponential decay of $ P_h(\Gah) \bar \delta_{q^{-1}(x_0)}$) and Lemma \ref{GFEst1} (local estimates of regularized Green's function), we have 
\begin{align}
\widehat{I_j}&=\|P_h(\Gah) \bar\delta_{q^{-1}(x_0)}\|_{1,\Gamma_j'}
+d_j^{-1}\| P_h(\Gah) \bar\delta_{q^{-1}(x_0)} \|_{\Gamma_j'} 
\leq Ch^2d_{j}^{-3-\dimsurf/2},  \label{EstIj}  \\[5pt]
\widehat{X_j}&=
d_j\vertiii{\partial_t(I_h\rg-\rg)}_{1,Q_j'}
+\vertiii{\partial_t(I_h\rg-\rg)}_{Q_j'}  \nn\\
&\quad\, 
+d_j^{-1}\vertiii{I_h\rg-\rg}_{1,Q_j'}
+ d_j^{-2}\vertiii{I_h\rg-\rg}_{Q_j'}  \nn\\
&
\leq Cd_j h \vertiii{\partial_{t}\rg}_{L^2H^{2}(Q_j'')}
+ C d_j ^{-1} h\vertiii{\rg}_{L^2H^{2}(Q_j'')} \nn\\
&
\leq C h d_j^{-2-\dimsurf/2}  , \\
 \widehat{Y_j} &=  h^{k+1}
\bigg(d_j \vertiii{\nabla_\Gamma \partial_t \rg}_{Q_j'}
+  \vertiii{\partial_t \rg}_{Q_j'}
+ d_j^{-1} \vertiii{\nabla_\Gamma \rg}_{Q_j'}
+ d_j^{-2} \vertiii{\rg}_{Q_j'}\bigg) \notag\\
&\leq C h^{k+1} d_j^{-1-\dimsurf/2} ,
\end{align}
and
\begin{align}
\overline{I_j}&=d_{j}^{2}\|\Delta_{\Gamma,h} (P_h(\Gah) \bar\delta_{q^{-1}(x_0)})^\ell \|_{1,\Gamma_j'}
+d_j \|\Delta_{\Gamma,h} (P_h(\Gah) \bar\delta_{q^{-1}(x_0)})^\ell \|_{\Gamma_j'}
\leq Ch^2d_{j}^{-3-\dimsurf/2},\\[5pt]
\overline{X_j}&=
d_j^3\vertiii{I_h\partial_{tt}\rg-\partial_{tt}\rg }_{1,Q_j'}
+d_j^2\vertiii{ I_h\partial_{tt}\rg-\partial_{tt}\rg}_{Q_j'} \nn\\
&\quad\, 
+d_j \vertiii{I_h\partial_{t}\rg-\partial_{t}\rg}_{1,Q_j'}+
\vertiii{I_h\partial_{t}\rg-\partial_{t}\rg}_{Q_j'}  \nn\\
&
\leq C ( d_j^3 h+ d_j^{ 2}h^{2})\vertiii{\partial_{tt}\rg}_{L^2H^{2}(Q_j'')}
+ C (d_j h+ h^{2})\vertiii{\partial_{t}\rg}_{L^2H^{2}(Q_j'')}  \nn\\
&
\leq C h d_j^{-2-\dimsurf/2} ,  \\
\label{ovXj}
\overline{Y_j} &=  h^{k+1}
\bigg(d_j^3 \vertiii{\nabla_\Gamma \partial_{tt} \rg}_{Q_j'}
+ d_j^2 \vertiii{\partial_{tt} \rg}_{Q_j'}
+ d_j \vertiii{\nabla_\Gamma \partial_t \rg}_{Q_j'}
+ \vertiii{\partial_t \rg}_{Q_j'}\bigg) \notag\\
&\leq C h^{k+1} d_j^{-1-\dimsurf/2} .
\end{align}
By substituting \eqref{F1tE}-\eqref{ovXj} into the expression of ${\mathscr K}$ in \eqref{express-K}, we have
\begin{align} 
{\mathscr K} 
&=\sum_{j} d_{j}^{1+\dimsurf/2} (d_j^{-1}\vertiii{F}_{1,Q_j}
+\vertiii{\partial_tF}_{Q_j}
+d_{j}\vertiii{\partial_tF}_{1,Q_j}
+d_{j}^2\vertiii{\partial_{tt}F}_{Q_j})  \nn\\
&\leq C \sum_{j}  d_{j}^{1+\dimsurf/2}
\epsilon^{-3}
\big(h^2d_j^{-3-\dimsurf/2}+h d_j^{-2-\dimsurf/2}
+d_j^{-2}\vertiii{F}_{Q'_j} \big) \nn\\
&\quad 
+ \sum_{j} (C h^{1/2}d_j^{-1/2} + C\epsilon^{-1}hd_j^{-1} + \epsilon)
d_{j}^{1+\dimsurf/2} 
(d_j^{-1}\vertiii{F}_{1,Q_j'} +\vertiii{\partial_tF}_{Q_j'}
+d_{j}\vertiii{\partial_tF}_{1,Q_j'}
+d_{j}^2\vertiii{\partial_{tt}F}_{Q_j'} ) \nn\\
&\leq C 
+C\epsilon^{-3}  \sum_{j}  d_{j}^{-1+\dimsurf/2} \vertiii{F}_{Q_j'}  \nn\\
&\quad 
+ \sum_{j}(C h^{1/2}d_j^{-1/2} + C\epsilon^{-1}hd_j^{-1} + \epsilon) d_{j}^{1+\dimsurf/2}
(d_j^{-1}\vertiii{F}_{1,Q_j'} +\vertiii{\partial_tF}_{Q_j'}
+d_{j}\vertiii{\partial_tF}_{1,Q_j'}
+d_{j}^2\vertiii{\partial_{tt}F}_{Q_j'} )  .
\end{align}

Since $\vertiii{F}_{Q_j'} \le C(\vertiii{F}_{Q_{j-1}} + \vertiii{F}_{Q_{j}} + \vertiii{F}_{Q_{j+1}})$, 
we can convert the $Q_j'$-norm in the inequality above to the $Q_j$-norm:
\begin{align}\label{slkll}
{\mathscr K} &\leq C 
+ C\epsilon^{-3} d_{*}^{-1+\dimsurf/2}  \vertiii{F}_{Q_*} 
+ C\epsilon^{-3} \sum_{j}  d_{j}^{-1+\dimsurf/2}  \vertiii{F}_{Q_j} \nn\\ 
&\quad 
+\sum_{j}
(C h^{1/2}d_j^{-1/2} + C\epsilon^{-1}hd_j^{-1} + \epsilon)
 d_*^{1+\dimsurf/2}
(d_*^{-1}\vertiii{F}_{1,Q_*} +\vertiii{\partial_tF}_{Q_*}
+d_{*}\vertiii{\partial_tF}_{1,Q_*}
+d_{*}^2\vertiii{\partial_{tt}F}_{Q_*} )  \nn\\ 
&\quad 
+\sum_{j} (C h^{1/2}d_j^{-1/2} + C\epsilon^{-1}hd_j^{-1} + \epsilon) d_{j}^{1+\dimsurf/2}
(d_j^{-1}\vertiii{F}_{1,Q_j} +\vertiii{\partial_tF}_{Q_j}
+d_{j}\vertiii{\partial_tF}_{1,Q_j}
+d_{j}^2\vertiii{\partial_{tt}F}_{Q_j} ) \nn\\
&\leq 
C_\epsilon
+C_\epsilon C_*^{3+\dimsurf/2}+ 
C_\epsilon \sum_{j} d_{j}^{-1+\dimsurf/2} \vertiii{F}_{Q_j} 
+ C (C_\epsilon C_*^{-1/2} +C C_*^{-1} + \epsilon) {\mathscr K} ,
\end{align}
where we have used $d_j\ge C_*h$ and \eqref{EstGamma-Q1}-\eqref{EstGamma-Q5} to estimate 
$$
	\vertiii{F}_{1,Q_*}, \qquad \vertiii{\partial_tF}_{Q_*}, \qquad \vertiii{\partial_tF}_{1,Q_*}, \qquad \text{and} \qquad 
	\vertiii{\partial_{tt}F}_{Q_*} ,
$$
and used the expression of ${\mathscr K}$ in \eqref{express-K} to bound the terms involving $Q_j$.  
Since $\epsilon_* = \epsilon + \epsilon^{-1}/C_*$, we can make $\epsilon_*$ sufficiently small by first choosing $\epsilon$ small enough and then choosing $C_*$ large enough ($\epsilon$ can be fixed now and $C_*$ will be determined later). Then the last term on the right-hand side of \eqref{slkll} can be absorbed by the left-hand side. Therefore, we obtain
\begin{align} \label{K-L2F}
{\mathcal K}
&\leq
C +C C_*^{3+\dimsurf/2}+
\sum_{j} d_{j}^{-1+\dimsurf/2}  \vertiii{F}_{Q_j} .
\end{align}
%
%

It remains to estimate $\vertiii{F}_{Q_j}$. We apply the parabolic duality argument: Let $w$ be the solution of the following backward parabolic equation on the domain $\mathcal Q$:
\begin{align}
	\left\{\begin{array}{ll}
			-\partial_t w - \Delta_\Gamma w = v, \notag\\
		w(1)=0 ,  \notag
	\end{array}\right.
\end{align}
with ${\rm supp}(v) \subseteq Q_j$ and $\| v \|_{L^2(Q_j)} = 1$.
After testing the above equation by $F$, we get
\begin{align}\label{eq:dual_test}
	[F, v] = (F(0), w(0)) + [F_t, w] + [\nabla_\Gamma F, \nabla_\Gamma w] ,
\end{align}
where
\begin{align}
	&(F(0), w(0)) 
	= (\bar\delta_{h,q^{-1}(x_0)}^\ell - \tilde\delta_{x_0}, w(0)) \notag\\
	&= (P_h \tilde\delta_{x_0} - \tilde\delta_{x_0}, w(0) - I_h w(0))
	+ (\bar\delta_{h,q^{-1}(x_0)}^\ell - \tilde\delta_{h, x_0}, w(0)) \notag\\
	&= (P_h \tilde\delta_{x_0}, w(0) - I_h w(0))_{\Gamma_j''}
	+ (P_h \tilde\delta_{x_0} - \tilde\delta_{x_0}, w(0) - I_h w(0))_{(\Gamma_j'')^c}
	+ (\bar\delta_{h,q^{-1}(x_0)}^\ell - \tilde\delta_{h, x_0}, w(0)) \notag\\
	&=: \mathcal I_1 + \mathcal I_2 + \mathcal I_3 .
\end{align} 
Both $\mathcal I_1$ and $\mathcal I_2$ can be estimated in the same way as \cite[Eq. (5.21), (5.23) and (5.24)]{Li-MCOM-2019}. The only difference is the replacement of flat domain $\Omega$ in \cite{Li-MCOM-2019} by surface $\Gamma$ here. These estimates of $\mathcal I_1$ and $\mathcal I_2$ can be written as follows:
\begin{align}
| \mathcal I_1 | \leq C h^2 d_j^{-1-\dimsurf/2} 
\quad\mbox{and}\quad\,\, 
	| \mathcal I_2 | 
	&\leq C h^2 \| w(0) \|_{W^{2,\infty}((\Gamma_j')^c)} , \notag\\
	&\leq C h^2 \sup_{y\in\Gamma} \| G(\cdot,\cdot,y) \|_{L^\infty W^{2,\infty}(\cup_{k\leq j+\log_2 C} Q_k)} \| v \|_{L^{1}(Q_j)} , \notag\\
	&\leq C h^2 d_j^{-2-\dimsurf} \| v \|_{L^{1}(Q_j)} , \notag\\
	&\leq C h^2 d_j^{-1-\dimsurf/2} \vertiii{v}_{L^{2}(Q_j)} , \notag
\end{align}
where in the second-to-last inequality we have used the following pointwise estimate of Green's function (cf.~\eqref{GausEst1}):
\begin{align}
	&\sup_{y\in\Gamma} \| G(\cdot,\cdot,y) \|_{L^\infty W^{2,\infty}(\cup_{k\leq j+\log_2 C} Q_k)} \notag\\
	&\leq C
	\sup_{y\in\Gamma} \| G(\cdot,\cdot,y) \|_{L^\infty L^{\infty}(\cup_{k\leq j+\log_2 C} Q_k)}^{1-\theta}
	\| G(\cdot,\cdot,y) \|_{L^\infty H^{M}(\cup_{k\leq j+\log_2 C} Q_k)}^\theta
	\notag\\
	&\leq
	C d_j^{-m(1-\theta)} d_j^{-(M+m/2)\theta}
	=
	C d_j^{-2-\dimsurf} , \notag
\end{align}
where in the first and second inequality we have used the Gagliardo–Nirenberg interpolation inequality with $\theta=\frac{4}{2M-m}$ for any integer $M>\frac{m}{2}$ (cf.~\cite[Theorem 3.70]{Aubin82}) and the local estimates of Green's function (Lemma \ref{GFEst1} and Remark \ref{rmk:high_est}) respectively.

Similarly, by using Lemma \ref{lemma:delta} and \eqref{GausEst1}, the argument in \cite{Li-MCOM-2019} leads to the following estimate:
\begin{align}
	| \mathcal I_3 | 
	&\leq 
	\| \bar\delta_{h,q^{-1}(x_0)}^\ell - \tilde\delta_{h, x_0} \|_{L^{1}(\Gamma)} \| w(0) \|_{L^{\infty}(\Gamma)} 
	\notag\\
	&\leq 
	C h^{k+1} \sup_{y\in\Gamma} \| G(\cdot,\cdot,y) \|_{L^\infty L^{\infty}(\cup_{k\leq j+\log_2 C} Q_k)} \| v \|_{L^{1}(Q_j)}
	\notag\\
	&\leq 
	C h^{k+1} d_j^{-\dimsurf} \| v \|_{L^{1}(Q_j)}
	\notag\\
	&\leq 
	C h^{k+1} d_j^{1-\dimsurf/2} \vertiii{v}_{L^{2}(Q_j)} . \notag
\end{align}

We decompose the second and third terms on the right-hand side of \eqref{eq:dual_test} as follows:
\begin{align}
	&[F_t, w] + [\nabla_\Gamma F, \nabla_\Gamma w]  \notag\\
	&= [\partial_t(\rg_h^\ell - \rg), w] + [\nabla_\Gamma (\rg_h^\ell - \rg), \nabla_\Gamma w]  \notag\\
	&= [a_h\partial_t\rg_h^\ell - \partial_t\rg, w] + [B_h\nabla_\Gamma \rg_h^\ell - \nabla_\Gamma \rg, \nabla_\Gamma w]  \notag\\
	&\quad+ [(1 - a_h)\partial_t \rg_h^\ell, w] + [(1 - B_h)\nabla_\Gamma \rg_h^\ell, \nabla_\Gamma w]  \notag\\
	&= [a_h\partial_t\rg_h^\ell - \partial_t\rg, (w - I_h w)] + [B_h\nabla_\Gamma \rg_h^\ell - \nabla_\Gamma \rg, \nabla_\Gamma (w - I_h w)]  \notag\\
	&\quad+ [(1 - a_h)\partial_t \rg, w] + [(1 - B_h)\nabla_\Gamma \rg, \nabla_\Gamma w]  \notag\\
	&\quad+ [(1 - a_h)\partial_t (\rg_h^\ell - \rg), w] + [(1 - B_h)\nabla_\Gamma (\rg_h^\ell - \rg), \nabla_\Gamma w]  \notag\\
	&= [a_h\partial_t(\rg_h^\ell - \rg), (w - I_h w)] + [B_h\nabla_\Gamma (\rg_h^\ell - \rg), \nabla_\Gamma (w - I_h w)]  \notag\\
	&\quad+ [(a_h - 1) \partial_t\rg, (w - I_h w)] + [(B_h - 1) \nabla_\Gamma \rg, \nabla_\Gamma (w - I_h w)]  \notag\\
	&\quad+ [(1 - a_h)\partial_t \rg, w] + [(1 - B_h)\nabla_\Gamma \rg, \nabla_\Gamma w]  \notag\\
	&\quad+ [(1 - a_h)\partial_t (\rg_h^\ell - \rg), w] + [(1 - B_h)\nabla_\Gamma (\rg_h^\ell - \rg), \nabla_\Gamma w]  \notag\\
	&=: \sum_{i=4}^{11} \mathcal I_i .
\end{align}
By \cite[Eq. (5.26)]{Li-MCOM-2019}, we know
\begin{align}
	| \mathcal I_4 | + | \mathcal I_5 | \leq C \sum_{*, i} (h^{2} \vertiii{F_t}_{Q_i} + h \vertiii{F}_{1, Q_i}) \|w \|_{L^2 H^{2}(Q_i')} .
\end{align}
Using the interpolation error estimate,
\begin{align}
	| \mathcal I_6 | + | \mathcal I_7 | 
	&\leq C h^{k+1}\sum_{*, i} (h^{2} \vertiii{\rg_t}_{Q_i} + h \vertiii{\rg}_{1, Q_i}) \|w \|_{L^2 H^{2}(Q_i')} , \notag
\end{align}
 and from Lemma \ref{lemma:geo_perturb}, Sobolev embedding and H\"older's inequality, it holds
\begin{align}
	| \mathcal I_8 | + | \mathcal I_9 | 
	&\leq C h^{k+1} \sum_{*, i} ( d_i^{2} \vertiii{\rg_t}_{Q_i} + d_i \vertiii{\rg}_{1, Q_i}) \|w \|_{L^2 H^{2}(Q_i')} , \notag\\
	| \mathcal I_{10} | + | \mathcal I_{11} | 
	&\leq C h^{k+1} \sum_{*, i} ( d_i^{2} \vertiii{F_t}_{Q_i} + d_i \vertiii{F}_{1, Q_i}) \|w \|_{L^2 H^{2}(Q_i')} . \notag
\end{align}
According to the estimates above and \cite[Eqs. (5.31)--(5.32)]{Li-MCOM-2019} (with $\alpha=1$ therein), we get
\begin{align}\label{FL2Qj} 
\vertiii{F}_{Q_j} 
 &\leq 
 Ch^{2}d_j^{-1-\dimsurf/2} 
 +C h^{k+1} \sum_{*,i}(d_i^{2}\vertiii{\rg_t}_{Q_i} 
 + d_i \vertiii{\rg}_{1,Q_i})
 \Big(\frac{\min (d_i ,d_j )}{\max (d_i ,d_j )}\Big)  \notag\\
 &\quad
 +C\sum_{*,i}(h^{2}\vertiii{F_t}_{Q_i} 
 + h \vertiii{F}_{1,Q_i})
 \Big(\frac{\min (d_i ,d_j )}{\max (d_i ,d_j )}\Big) 
 \notag\\
 &\leq Ch^{2}d_j^{-1-\dimsurf/2} 
 +C h^{k+1} \sum_{*,i}
d_i^{1-\dimsurf/2}
\Big(\frac{\min (d_i ,d_j )}{\max (d_i ,d_j )}\Big)
 \notag\\
 &\quad
 +C\sum_{*,i}(h^{2}\vertiii{F_t}_{Q_i}  + h \vertiii{F}_{1,Q_i})
 \Big(\frac{\min (d_i ,d_j )}{\max (d_i ,d_j )}\Big) 
  .
\end{align}
Note that (cf.~\cite[Eq. (5.33)]{Li-MCOM-2019} with $\alpha=1$ therein)
\begin{align}\label{lalnu}
\sum_j d_{j}^{-1+\dimsurf/2} 
\Big(\frac{\min (d_i ,d_j )}{\max (d_i ,d_j )}\Big)
\leq 
Cd_i^{-1-\dimsurf/2} .
\end{align}
By substituting \eqref{FL2Qj}--\eqref{lalnu} into \eqref{K-L2F} we obtain 
\begin{align*}
{\mathscr K} &\leq 
C+CC_*^{3+\dimsurf/2}+
\sum_{j} Cd_{j}^{-1+\dimsurf/2} \vertiii{F}_{Q_j} \nn\\
&\leq C+CC_*^{3+\dimsurf/2}
+C \sum_j \left(\frac{h}{d_j}\right)^{2} 
\qquad\mbox{(here we have used \eqref{FL2Qj})}\\
&\quad
+C\sum_j d_{j}^{-1+\dimsurf/2} h^{k+1}
\sum_{*,i} d_i^{1-\dimsurf/2}
 \Big(\frac{\min (d_i ,d_j )}{\max (d_i ,d_j )}\Big) \notag\\
&\quad
+C\sum_j  d_{j}^{-1+\dimsurf/2} 
\sum_{*,i}(h^{2}\vertiii{F_t}_{Q_i} 
+ h \vertiii{F}_{1,Q_i})
\Big(\frac{\min (d_i ,d_j )}{\max (d_i ,d_j )}\Big) . 
\end{align*}
Consequently,
\begin{align*}
{\mathscr K}  
&\leq
C+CC_*^{3+\dimsurf/2} +C  C_*^{-2}
\qquad\mbox{(here we exchange the order of summation)} \\
&\quad
+C\sum_{*,i} h^{k+1} d_i^{1-\dimsurf/2}
\sum_{j}
d_{j}^{-1+\dimsurf/2}
 \Big(\frac{\min (d_i ,d_j )}{\max (d_i ,d_j )}\Big) \notag\\
&\quad 
+C\sum_{*,i}(h^{2}\vertiii{F_t}_{Q_i} 
+ h \vertiii{F}_{1,Q_i})
\sum_j d_{j}^{-1+\dimsurf/2} 
\Big(\frac{\min (d_i ,d_j )}{\max (d_i ,d_j )}\Big) \\
&\leq
C+CC_*^{3+\dimsurf/2} + C  C_*^{-2}
+C\sum_{*,i} h^{k+1} \notag\\
&\quad
+C\sum_{*,i}(h^{2}\vertiii{F_t}_{Q_i} 
+ h \vertiii{F}_{1,Q_i})d_i^{\frac{\dimsurf}{2}-1} 
\qquad\mbox{(here we use \eqref{lalnu})}\\
&=
C+CC_*^{3+\dimsurf/2}+C  C_*^{-2} 
+ C(\log C_* + \log\frac{1}{h}) h^{k+1}
\notag\\
&\quad
+C\sum_{*,i}d_i^{1+\dimsurf/2}\Big(\vertiii{F_t}_{Q_i} 
\Big(\frac{h}{d_i}\Big)^{2} 
+ d_i^{-1}\vertiii{F}_{1,Q_i}\Big(\frac{h}{d_i}\Big) \Big) \\
&\leq 
C+CC_*^{3+\dimsurf/2}+C  C_*^{-2}
+ C(\log C_* + \log\frac{1}{h}) h^{k+1}
+Cd_{*}^{1+\dimsurf/2}
\left(\vertiii{F_t}_{Q_*} + d_j^{-1}\vertiii{F}_{1,Q_*} \right) \\
&\quad
+C\sum_{i} d_i^{1+\dimsurf/2}\left(\vertiii{F_t}_{Q_i} 
+ d_j^{-1}\vertiii{F}_{1,Q_i}\right)\bigg(\frac{h}{d_i}\bigg) \\
&\leq 
C+CC_*^{3+\dimsurf/2}+C  C_*^{-2}
+ C\Big(\log C_* + \log\frac{1}{h}\Big) h^{k+1}
+C C_*^{-1} {\mathscr K}   .
\end{align*}
By choosing $C_*$ to be large enough ($C_*$ is determined now), the term $C C_*^{-1} {\mathscr K}$ will be absorbed by the left-hand side of the inequality above. In this case, the inequality above implies 
\begin{align}
{\mathscr K} \le C .
\end{align}  
Substituting the last inequality into \eqref{Bd31K2} yields 
\begin{align}\label{L1FtQ1}
\| \partial_t F\|_{L^1((0,1)\times \Gamma)} + \| t\partial_{tt} F\|_{L^1((0,1)\times \Gamma)} 
&\leq C .
\end{align} 

{\it Part II.}$\,\,\,$ 
For $t\in[1,+\infty)$, we first define the finite element space with vanishing average:
$\mathring S_h(\Gah) = \{\varphi_h \in S_h(\Gah): \int_{\Gah} \varphi_h =0 \}$
and
$\mathring S_h(\Gamma) = \{\varphi_h^\ell \in S_h(\Gamma): \int_{\Gamma} \varphi_h^\ell =0 \}$ on which we can define the inverse of the discrete Laplacian operators $\Delta_{\Gah,h}^{-1}$ and $\Delta_{\Gamma,h}^{-1}$ respectively.
\begin{lemma}
	For any $f_h \in \mathring S_h(\Gah)$, we define $f_h^{\rm avg, \ell}  = f_h^\ell -  |\Gamma|^{-1} \int_{\Gamma} f_h^\ell\in \mathring S_h(\Gamma)$. Then we have the following properties
	\begin{align}
		\| (\Delta_{\Gamma_h,h}^{-1} f_h)^\ell \|_{L^2(\Gamma)}
		&\sim
		\| \Delta_{\Gamma,h}^{-1} f_h^{\rm avg, \ell} \|_{L^2(\Gamma)} , \notag\\
		\| \Delta_{\Gamma,h}^{-1} f_h^{\rm avg, \ell} \|_{L^2(\Gamma)}
		&\leq
		C \| f_h^{\rm avg, \ell} \|_{L^1(\Gamma)} 
		\leq
		2C \| f_h^\ell \|_{L^1(\Gamma)} . \notag
	\end{align}
\end{lemma}
\begin{proof}
	The first result is a consequence of consistency of $\Gamma$ and $\Gah$. We define $u_h = \Delta_{\Gah,h}^{-1} f_h \in \mathring S_h(\Gah)$ and $\tilde u_h^\ell = \Delta_{\Gamma,h}^{-1} f_h^{\rm avg, \ell} \in \mathring S_h(\Gamma)$ with $\int_{\Gah} u_h = \int_{\Gamma } \tilde u_h^\ell =0$ and define 
	$$u_h^{\rm avg, \ell}  = u_h^\ell -  |\Gamma|^{-1} \int_{\Gamma} u_h^\ell\in \mathring S_h(\Gamma) .$$ 
	Since $\int_{\Gamma} u_h^\ell = \int_{\Gamma_h}(a_h^{-1})^{-\ell} u_h = \int_{\Gamma_h}((a_h^{-1})^{-\ell}-1) u_h$ and $|(a_h^{-1})^{-\ell}-1|\le Ch^{k+1}$, it follows that 
	\begin{align}\label{eq:avg1}
		| u_h^{\rm avg, \ell}- u_h^\ell | \leq C h^{k+1} \| u_h \|_{L^1(\Gah)} ,
	\end{align}
	and, similarly 
	\begin{align}\label{eq:avg2}
		| f_h^{\rm avg, \ell}- f_h^\ell | \leq C h^{k+1} \| f_h \|_{L^1(\Gah)} .
	\end{align}
	By definition,
	\begin{align}
		(\nabla_{\Gah} u_h, \nabla_{\Gah} v_h)_{\Gah}
		&= (f_h ,v_h)_{\Gah} \quad\forall v_h\in S_h , \notag\\
		(\nabla_{\Gamma} \tilde u_h^\ell, \nabla_{\Gamma} v_h^\ell)_{\Gamma}
		&= (f_h^{\rm avg, \ell} ,v_h^\ell)_{\Gamma} \quad\forall v_h\in S_h \notag .
	\end{align}
	Applying change of variables and subtraction, we obtain
	\begin{align}
		(\nabla_{\Gamma} (u_h^\ell - \tilde u_h^\ell), \nabla_{\Gamma} v_h^\ell)_{\Gamma}
		&= 
		-((B_h-1) \nabla_{\Gamma} \tilde u_h^\ell, \nabla_{\Gamma} v_h^\ell)_{\Gamma}
		+
		((a_h-1) f_h^\ell ,v_h^\ell)_{\Gamma} 
		+
		( (f_h^\ell - f_h^{\rm avg, \ell}) ,v_h^\ell)_{\Gamma} . \notag
	\end{align}
	Then we test with $v_h^\ell = u_h^{\rm avg, \ell} - \tilde u_h^\ell \in \mathring S_h(\Gamma)$ and use \eqref{eq:avg1}--\eqref{eq:avg2}, Lemma \ref{lemma:geo_perturb}, Poincar\'e inequality and the inverse inequality on $\Gamma$ to derive
	\begin{align}
		&\| \nabla_{\Gamma} (u_h^{\rm avg, \ell} - \tilde u_h^\ell) \|_{L^2(\Gamma)}
		=
		\| \nabla_{\Gamma} (u_h^{\ell} - \tilde u_h^\ell) \|_{L^2(\Gamma)} \notag\\
		&\leq
		C h^{k+1} \| \nabla_{\Gamma} \tilde u_h^\ell \|_{L^2(\Gamma)}
		+
		C h^{k+1} \| f_h^\ell \|_{L^2(\Gamma)} \notag\\
		&\leq
		C h^{k+1} \| \nabla_{\Gamma} \tilde u_h^\ell \|_{L^2(\Gamma)}
		+
		C h^{k+1} \min \{\| f_h^\ell \|_{L^2(\Gamma)}, \| f_h^{\rm avg, \ell} \|_{L^2(\Gamma)}  \}
		\notag\\
		&\leq
		C h^{k} \|  \tilde u_h^\ell \|_{L^2(\Gamma)}
		+
		C h^{k-1} \min \{\| u_h^\ell \|_{L^2(\Gamma)}, \| \tilde u_h^\ell \|_{L^2(\Gamma)} \} , \notag
	\end{align}
	and from the triangle inequality and Poincar\'e inequality again
	\begin{align}
		\| u_h^\ell - \tilde u_h^\ell \|_{L^2(\Gamma)}
		&\leq
		\| u_h^{\rm avg, \ell} - u_h^\ell \|_{L^2(\Gamma)}
		+
		C \| \nabla_{\Gamma} (u_h^{\rm avg, \ell} - \tilde u_h^\ell) \|_{L^2(\Gamma)} \notag\\
		&\leq
		C h^{k+1} \| u_h^\ell \|_{L^1(\Gamma)}
		+
		C h^{k} \| \tilde u_h^\ell \|_{L^2(\Gamma)}
		+
		C h^{k-1} \min \{\| u_h^\ell \|_{L^2(\Gamma)}, \| \tilde u_h^\ell \|_{L^2(\Gamma)} \} . \notag
	\end{align}
	Therefore $\| \tilde u_h^\ell \|_{L^2(\Gamma)} \sim \| u_h^\ell \|_{L^2(\Gamma)}$
	, and this proves the first result.
	
	The second result follows from \cite[Lemma 4.3]{BLK23} 
	where we use the elliptic regularity theory and the $W^{1,q}$-stability of Ritz projection on the smooth surface $\Gamma$ (see Section \ref{sec:Rh}).
\end{proof}

Denote by $e^{t\Delta_{\Gah,h}}$ the analytical semigroup generated by the linear operator $\Delta_{\Gah,h}\colon  \mathring S_h(\Gah) \rightarrow \mathring S_h(\Gah)$. Then, by the previous lemma with $f_h$ replaced by $\bar\delta_{h,x_0}^{\rm avg}$ therein for any fixed $x_0\in \Gah$, if we define $\bar\delta_{h,x_0}^{\rm avg} = \bar\delta_{h,x_0} -  |\Gamma_h|^{-1} \int_{\Gah} \bar\delta_{h,x_0} \in \mathring S_h(\Gah)$ then
\begin{align}
	\| \partial_t \rg_h(t,\cdot, x_0) \|_{L^2(\Gah)}
	&=
	\| \partial_t e^{t\Delta_{\Gah,h}} \bar\delta_{h,x_0}^{\rm avg} \|_{L^2(\Gah)} \notag\\
	&=
	t^{-2}\| (t\Delta_{\Gah,h})^2 e^{t\Delta_{\Gah,h}} \Delta_{\Gah,h}^{-1} \bar\delta_{h,x_0}^{\rm avg} \|_{L^2(\Gah)} \notag\\
	&=
	t^{-2}\| \Delta_{\Gah,h}^{-1} \bar\delta_{h,x_0}^{\rm avg} \|_{L^2(\Gah)} \notag\\
	&\leq
	C t^{-2}\| \bar\delta_{h,x_0}^{\rm avg} \|_{L^1(\Gah)} \notag\\
	&\leq C t^{-2} . \notag
\end{align}

From the norm equivalence on $\Gamma$ and $\Gah$, it is straightforward to show the constant of Poincar\'e inequality (with vanishing average) on $\Gah$ is bounded from above by the constant of Poincar\'e inequality (with vanishing average) on $\Gamma$. As a consequence, the smallest eigenvalue of $-\Delta_{\Gah}\colon  \mathring S_h(\Gah) \rightarrow \mathring S_h(\Gah)$ has a lower bound, denoted by $\lambda_0$, which is uniform w.r.t $h$ and only depends on $\Gamma$. Hence (the average vanishes due to the differential $\partial_t$),
\begin{align}
	\| \partial_t \rg_h(t,\cdot,x_0) \|_{L^2(\Gah)}
	\leq 
	C e^{-\lambda_0 (t - 1)} \| \partial_t \rg_h(1,\cdot,x_0) \|_{L^2(\Gah)} 
	\leq 
	C e^{-\lambda_0 (t - 1)} . \notag
\end{align}
Similarly, we also have
\begin{align}
	\| \partial_{tt} \rg_h(t,\cdot,x_0) \|_{L^2(\Gah)}
	+
	\| \partial_t \rg(t,\cdot,x_0) \|_{L^2(\Gah)}
	+
	\| \partial_{tt} \rg(t,\cdot,x_0) \|_{L^2(\Gah)}
	\leq 
	C e^{-\lambda_0 (t - 1)} \quad t\geq 1 . \notag
\end{align}


The proof of Lemma \ref{LemGm2} is complete.

\section{Perturbation arguments for an evolving surface\\
(Proof of Theorem \ref{thm:disc_max_reg})}
\label{section:perturbation arguments for evolving surfaces}
\setcounter{equation}{0}
Following usual notational conventions, we will always identify a finite element function and the  corresponding vector collecting its nodal values. Such representation is unique if we have specified the underlying domain. For example, given any $t,s \in [0,T]$, the two integrands of $$\int_{\Gamma_h(t)} v_h \quad\mbox{and}\quad \int_{\Gamma_h(s)} v_h$$ have the same vector of nodal values, denoted by $\v$, but are defined on different domains $\Gamma_h(t)$ and $\Gamma_h(s)$. When the underlying domain is specified, $\v$ is automatically substantialized to a finite element function $v_h$ on that domain. Since all of the quantitative computations in this paper involve either integrals or norms, our notations for finite element functions will always have a unique and clear meaning. For another example, $\| v_h \|_{\Gamma_h(t)}$ and $\| v_h \|_{\Gamma_h(s)}$ denote the norms of a finite element function (a nodal vector) on the two different surfaces $\Gamma_h(t)$ and $\Gamma_h(s)$, respectively.
Correspondingly, the interpolation operator $I_h$ should be interpreted as the determination of the nodal vector which uniquely corresponds to a finite element function after specifying the underlying surface. 

Given any $s, t \in [0, T]$, we first pull back the scheme onto the piecewise polynomial approximate surface $\ls(s)$ and then carry out temporal perturbation argument on $\Gamma_h(s)$. 

For the fixed time $s\in [0, T]$,
we define the scalar function $a(t,x)$ and the $\R^{(\dimsurf+1)\times (\dimsurf+1)}$-valued function $B(t,x)$ with $x\in \Gamma(s)$ to be the smooth prefactor of the following change of variables via the smooth flow parametrized on $\Gamma(s)$, i.e.~$F_s(t)\colon  \Gamma(s) \rightarrow \Gamma(t)$,
\begin{align}
	\int_{\Gamma(t)} u v
	&=
	\int_{\Gamma(s)} a(t, x) (u \circ F_s) \, (v \circ F_s) ,
	\notag\\
	\int_{\Gamma(t)} \nabla_{\Gamma(t)} u \cdot \nabla_{\Gamma(t)} v
	&=
	\int_{\Gamma(s)} B(t, x) \nabla_{\Gamma(s)} (u \circ F_s) \cdot \nabla_{\Gamma(s)} (v \circ F_s) . \notag
\end{align}
Similarly, we define the piecewise smooth scalar function $\bar a(t,x)$ and matrix-valued function $\bar B(t,x)$ with $x\in \ls(s)$ to be the prefactor of the following change of variables via the smooth flow parametrized on $\ls(s)$, i.e.~$F_s(t) \circ q(s)\colon  \ls(s) \rightarrow \Gamma(t)$,
\begin{align}
	\int_{\Gamma(t)} u v
	&=
	\int_{\ls(s)} \bar a(t, x) (u \circ F_s \circ q(s)) \, (v \circ F_s \circ q(s) ),
	\notag\\
	\int_{\Gamma(t)} \nabla_{\Gamma(t)} u \cdot \nabla_{\Gamma(t)} v
	&=
	\int_{\ls(s)} \bar B(t, x) \nabla_{\ls(s)} (u \circ F_s \circ q(s)) \cdot \nabla_{\ls(s)} (v \circ F_s \circ q(s)) , \notag
\end{align}
and define the piecewise smooth $\bar a_h(t,x)$ and $\bar B_h(t,x)$ with $x\in\ls(s)$ to be the prefactor of the following change of variables via the discrete flow parametrized on $\ls(s)$, i.e.~$F_{h,s}(t) := I_h (F_s(t)\circ q(s))\colon  \ls(s) \rightarrow \Gamma_h(t)$,
\begin{align}
	\int_{\Gamma_h(t)} u v
	&=
	\int_{\ls(s)} \bar a_h(t, x) u \circ F_{h,s} \, v \circ F_{h,s} ,
	\notag\\
	\int_{\Gamma_h(t)} \nabla_{\Gamma_h(t)} u \cdot \nabla_{\Gamma_h(t)} v
	&=
	\int_{\ls(s)} \bar B_h(t, x) \nabla_{\ls(s)} (u \circ F_{h,s}) \cdot \nabla_{\ls(s)} (v \circ F_{h,s}) . \notag
\end{align}
Since $F_{h,s}(t) = I_h (F_s(t)\circ q(s))$, by the interpolation error estimates, it follows that for any $s,t\in[0,T]$
\begin{align}
	\| \bar a_h(t, \cdot) - \bar a(t, \cdot) \|_{L^\infty(\ls(s))} &\leq C h^{k+1} , \notag\\
	\| \bar B_h(t, \cdot) - \bar B(t, \cdot) \|_{L^\infty(\ls(s))} &\leq C h^{k+1} \notag ,
\end{align}
and moreover from the geometric perturbation error estimates (cf.~\cite{Dem09})
\begin{align}
	\| \bar a^\ell(t, \cdot) - a(t, \cdot) \|_{L^\infty(\Gamma(s))} &\leq C h^{k+1} , \notag\\
	\| \bar B^{l}(t, \cdot) - B(t, \cdot) \|_{L^\infty(\Gamma(s))} &\leq C h^{k} \notag .
\end{align}
Hence $\tilde B(t,x) := (B(t)\circ q(s)) (x) \approx \bar B_h(t,x), x\in \ls(s), s,t\in [0, T]$ is a good approximation which is globally continuous and piecewise smooth. This global continuity will allow us to take advantage of the definition of discrete Laplacian. With the definitions above we have, for any $v_h\in S_h$,
\begin{align}\label{eq:scheme_pb0}
	&\quad(\partial_t u_h, v_h)_{\Gamma_h(s)}
	+ (\nabla_{\Gamma_h(s)} u_h, \nabla_{\Gamma_h(s)} v_h)_{\Gamma_h(s)} \notag\\
	&=(\bar a_h(s,\cdot) \partial_t u_h(t,\cdot), v_h)_{\ls(s)}
	+ (\bar B_h(s,\cdot) \nabla_{\ls(s)} u_h(t,\cdot), \nabla_{\ls(s)} v_h)_{\ls(s)} \notag\\
	&=
	(\bar a_h(t,\cdot) f_h(t,\cdot), v_h)_{\ls(s)}
	+
	((\bar a_h(s,\cdot) - \bar a_h(t,\cdot)) \partial_t u_h(t,\cdot), v_h )_{\ls(s)} \notag\\
	&\quad
	+ ((\bar B_h(s,\cdot) - \bar B_h(t,\cdot)) \nabla_{\ls(s)} u_h(t,\cdot), \nabla_{\ls(s)} v_h)_{\ls(s)} \notag\\
	&=:
	(\bar a_h(t,\cdot) f_h(t,\cdot), v_h)_{\ls(s)}
	+
	((\bar a_h(s,\cdot) - \bar a_h(t,\cdot)) \partial_t u_h(t,\cdot), v_h )_{\ls(s)}
	\notag\\
	&\quad+ (\Delta_{\ls(s),h}^{B_t} u_h(t,\cdot), v_h)_{\ls(s)} ,
\end{align}
where $\Delta_{\ls(s),h}^{B_t} u_h \in S_h(\ls(s))$ is defined as the Riesz representative of the following linear functional on $S_h(\ls(s))$:
\begin{align}
	((\bar B_h(s) - \bar B_h(t)) \nabla_{\ls(s)} u_h(t), \nabla_{\ls(s)} \cdot)_{\ls(s)}
	\colon S_h(\ls(s)) \rightarrow S_h(\ls(s)) . \notag
\end{align}
The estimate of $\Delta_{\ls(s),h}^{B_t} u_h$ is given in the lemma below.
\begin{lemma}\label{lemma:delta_h}
	For any $q\in [1,\infty]$, we have the estimate
	\begin{align}
		\| \Delta_{\ls(s),h}^{B_t} u_h(t) \|_{L^q(\Gamma_h(s))}
		\leq
		C\| \nabla_{\Gamma_h(s)} u_h(t) \|_{L^q(\Gamma_h(s))}
		+
		C |s-t|\| \Delta_{\Gamma_h(s),h} u_h(t) \|_{L^q(\Gamma_h(s))} . \notag
	\end{align}
\end{lemma}
\begin{proof}
%
	We apply change of variables, Leibniz rule, super-approximation estimate (cf.~(P3) in Section \ref{sec:hypo}) and the inverse inequality to get
	\begin{align}
		&\quad(\Delta_{\ls(s),h}^{B_t} u_h, v_h)_{\ls(s)} \notag\\
		&=
		(\nabla_{\Gamma_h(s)} u_h,  (\tilde B(s) - \tilde B(t)) \nabla_{\Gamma_h(s)} v_h)_{\Gamma_h(s)} \notag\\
		&\quad +
		\bigg(\nabla_{\Gamma_h(s)} u_h, \bigg((\bar B_h(s) - \bar B_h(t)) - (\tilde B(s) - \tilde B(t)) \bigg) \nabla_{\Gamma_h(s)} v_h \bigg)_{\Gamma_h(s)} \notag\\
		&= 
		- (\nabla_{\Gamma_h(s)} u_h, \nabla_{\Gamma_h(s)} (\tilde B(s) - \tilde B(t))  v_h)_{\Gamma_h(s)} 
		+ 
		(\nabla_{\Gamma_h(s)} u_h, \nabla_{\Gamma_h(s)} I_h ( (\tilde B(s) - \tilde B(t)) v_h))_{\Gamma_h(s)} \notag\\
		&\quad+ 
		(\nabla_{\Gamma_h(s)} u_h, \nabla_{\Gamma_h(s)} (1 - I_h) ((\tilde B(s) - \tilde B(t)) v_h))_{\Gamma_h(s)} \notag\\
		&\quad +
		\bigg(\nabla_{\Gamma_h(s)} u_h, \bigg((\bar B_h(s) - \bar B_h(t)) - (\tilde B(s) - \tilde B(t)) \bigg) \nabla_{\Gamma_h(s)} v_h \bigg)_{\Gamma_h(s)} \notag\\
		&\leq
		C \| \nabla_{\Gamma_h(s)} u_h \|_{L^p(\Gamma_h(s))} \| v_h \|_{L^q(\Gamma_h(s))} 
		+
		C |s-t| \| \Delta_{\Gamma_h(s),h} u_h \|_{L^p(\Gamma_h(s))} \| v_h \|_{L^q(\Gamma_h(s))}
		\notag\\
		&\quad
		+
		C (h + h^k) \| \nabla_{\Gamma_h(s)} u_h \|_{L^p(\Gamma_h(s))} \| \nabla_{\Gamma_h(s)} v_h \|_{L^q(\Gamma_h(s))}
		 \notag\\
		&\leq
		C \big(\| \nabla_{\Gamma_h(s)} u_h \|_{L^p(\Gamma_h(s))}
		+
		|s-t| \| \Delta_{\Gamma_h(s),h} u_h \|_{L^p(\Gamma_h(s))}\big) 
		\| v_h \|_{L^q(\Gamma_h(s))}
		\quad\forall v_h \in S_h , \notag
	\end{align}
	where in the second to last inequality we have used the consistency estimate $\| \bar B_h(t) - \tilde B(t) \|_{L^\infty(\Gamma_h(s))} \leq C h^k$ for any $t\in [0, T]$, and all of the constants are independent of $s,t \in [0, T]$ (possibly depend on $T$). Thus we conclude the lemma by duality.
\end{proof}
\begin{remark}\label{rmk:delta_h}
	The above proof additionally leads to, for any $s \in [0,T]$,
	\begin{align}
		\| \Delta_{\Gamma(s),h} u_h \|_{L^{q}(\Gamma(s))}
		\leq
		C \| \nabla_{\Gamma_h(s)} u_h \|_{L^q(\Gamma_h(s))}
		+
		C \| \Delta_{\Gamma_h(s),h} u_h \|_{L^q(\Gamma_h(s))} , \notag
	\end{align}
	and, for any $s^\prime \in [0,T]$,
	\begin{align}
		\| \Delta_{\Gamma_h(s), h} u_h \|_{L^{q}(\Gamma_h(s))}
		\leq
		C \| \nabla_{\Gamma_h(s^\prime),h} u_h \|_{L^q(\Gamma_h(s^\prime))}
		+
		C \| \Delta_{\Gamma_h(s^\prime),h} u_h \|_{L^q(\Gamma_h(s^\prime))} . \notag
	\end{align}
\end{remark}
\begin{lemma}\label{lemma:interp}
	The following discrete interpolation inequality holds on the polynomial surface $\Gamma_h(s)$, for all $q\in (1,\infty)$,
	\begin{align}
		\| \nabla_{\Gamma_h(s)} u_h \|_{L^q(\Gamma_h(s))}
		&\leq C
		\epsilon^{-1} \| u_h \|_{L^q(\Gamma_h(s))}
		+
		C \epsilon
		\| \Delta_{\Gamma_h(s),h} u_h \|_{L^q(\Gamma_h(s))}
		. \notag
	\end{align}
\end{lemma}
\begin{proof}
	We define function $u$ to be the solution of the following elliptic equation on $\Gamma(s)$
	\begin{align}
		-(\Delta_{\Gamma(s)} - 1)u = -(\Delta_{\Gamma(s),h} - 1)u_h^\ell , \notag
	\end{align}
	or equivalently $R_h u = u_h^\ell$ where $R_h$ is the Ritz projection associated to the finite element space $S_h(\Gamma(s))$ the elliptic operator $-\Delta_{\Gamma(s)} + 1$ (see Section \ref{sec:Rh}).
	Then the elliptic regularity theory says
	\begin{align}\label{eq:interp_2}
		\| u \|_{W^{2,q}(\Gamma(s))} \leq C \| (\Delta_{\Gamma(s),h} - 1) u_h^\ell \|_{L^q(\Gamma(s))} .
	\end{align}
	Moreover from the inverse inequality
	\begin{align}\label{eq:interp_3}
		\| (\Delta_{\Gamma(s),h} - 1) u_h^\ell \|_{L^q(\Gamma(s))}
		&=
		\sup_{v_h\in S_h, \| v_h^\ell \|_{L^{q^\prime}(\Gamma(s))} = 1}
		(\nabla_{\Gamma(s)} u_h^\ell, \nabla_{\Gamma(s)} v_h^\ell)_{\Gamma(s)}
		+
		(u_h^\ell, v_h^\ell)_{\Gamma(s)}
		\notag\\
		&\leq
		C h^{-2} \| u_h^\ell \|_{L^q(\Gamma(s))} .
	\end{align}
	Define $A = -\Delta_{\Gamma(s)} + 1$ and we get
	\begin{align}
		&(u, v)_{\Gamma(s)} = (A u, A^{-1} v)_{\Gamma(s)} = ((-\Delta_{\Gamma(s),h} + 1) u_h^\ell, A^{-1} v)_{\Gamma(s)} \notag\\
		&=
		((-\Delta_{\Gamma(s),h} + 1) u_h^\ell, R_h A^{-1} v)_{\Gamma(s)}
		+
		((-\Delta_{\Gamma(s),h} + 1) u_h^\ell, (1 - R_h) A^{-1} v)_{\Gamma(s)} \notag\\
		&=
		((\nabla_{\Gamma(s),h} + 1) u_h^\ell, (\nabla_{\Gamma(s),h} + 1) A^{-1} v)_{\Gamma(s)}
		+
		((-\Delta_{\Gamma(s),h} + 1) u_h^\ell, (1 - R_h) A^{-1} v)_{\Gamma(s)} \notag\\
		&=
		(u_h^\ell, v)_{\Gamma}
		+
		((-\Delta_{\Gamma(s),h} + 1) u_h^\ell, (1 - R_h) A^{-1} v)_{\Gamma(s)} , \notag
	\end{align}
	and therefore 
	\begin{align}\label{eq:interp_4}
		&\| u \|_{L^q(\Gamma(s))}
		=
		\sup_{\| v \|_{L^{q^\prime}(\Gamma(s))} = 1}
		(u, v)_{\Gamma(s)} \notag\\
		&\leq
		\| u_h^\ell \|_{L^q(\Gamma(s))}
		+ \| (-\Delta_{\Gamma(s),h} + 1) u_h^\ell \|_{L^{q}(\Gamma(s))} \sup_{\| v \|_{L^{q^\prime}(\Gamma(s))} = 1} \| (1 - R_h) A^{-1} v \|_{L^{q^\prime}(\Gamma(s))}
		\notag\\
		&\leq
		\| u_h^\ell \|_{L^q(\Gamma(s))}
		+ C h^2 \| (-\Delta_{\Gamma(s),h} + 1) u_h^\ell \|_{L^{q}(\Gamma(s))} \sup_{\| v \|_{L^{q^\prime}(\Gamma(s))} = 1} \| A^{-1} v \|_{W^{2,q^\prime}(\Gamma(s))}
		\notag\\
		&\leq
		C \| u_h^\ell \|_{L^q(\Gamma(s))} ,
	\end{align}
	where in the last line we have used \eqref{eq:interp_3} and the elliptic regularity theory $\| A^{-1} v \|_{W^{2,q^\prime}(\Gamma(s))} \leq C \| v \|_{L^{q^\prime}(\Gamma(s))} = C$.
	
	Hence the norm equivalence, Remark \ref{rmk:delta_h}, \eqref{eq:interp_2}, \eqref{eq:interp_4} and the $W^{1,q}$-stability of $R_h$ (see Section \ref{sec:Rh}) imply
	\begin{align}
		\| \nabla_{\Gamma(s)} u_h^\ell \|_{L^{q}(\Gamma(s))}
		&\leq 
		\| \nabla_{\Gamma(s)} (R_h u - I_h u) \|_{L^{q}(\Gamma(s))}
		+
		\| \nabla_{\Gamma(s)} (I_h - 1) u \|_{L^{q}(\Gamma(s))}
		+
		\| \nabla_{\Gamma(s)} u \|_{L^{q}(\Gamma(s))}
		\notag\\
		&\leq 
		C h \| u \|_{W^{2,q}(\Gamma(s))}
		+
		\| u \|_{L^{q}(\Gamma(s))}^{1/2} \| u \|_{W^{2,q}(\Gamma(s))}^{1/2} \notag\\
		&\leq 
		C \epsilon^{-1} \| u_h \|_{L^{q}(\Gamma_h(s))}
		+
		C \epsilon \| \nabla_{\Gahs} u_h \|_{L^{q}(\Gamma_h(s))}
		+
		C \epsilon \| \Delta_{\Gamma_h(s),h} u_h \|_{L^q(\Gamma_h(s))} . \notag
	\end{align}
	We complete the proof by absorbing $\epsilon \| \nabla_{\Gahs} u_h \|_{L^{q}(\Gamma_h(s))}$ into the left-hand side.
\end{proof}

Applying the discrete maximal regularity on the stationary surface $\Gamma_h(s)$ with $p,q\in(1,\infty)$ (Theorem \ref{eq:scheme_stat}) to \eqref{eq:scheme_pb0} and using the norm equivalence between $\Gamma_h(s)$ and $\ls(s)$, we get
\begin{align}
	&\quad\| \partial_t u_h \|_{L^p(0,t; L^q(\Gamma_h(s)))} 
	+ 
	\| \Delta_{\Gamma_h(s),h} u_h \|_{L^p(0,t; L^{q}(\Gamma_h(s)))} \notag\\
	&\leq
	C \| f_h \|_{L^p(0,t; L^q(\ls(s)))}
	+
	C \| (\bar a_h(s,\cdot) - \bar a_h(t,\cdot)) \partial_t u_h(t,\cdot) \|_{L^p(0,t; L^q(\ls(s)))}
	\notag\\
	&\quad
	+ C \| \Delta_{\ls(s),h}^{B_t} u_h(t,\cdot) \|_{L^p(0,t; L^q(\ls(s)))} \notag\\
	&\leq
	C \| f_h \|_{L^p(0,t; L^q(\Gamma_h(s)))}
	+
	C |s-t|\| \partial_t u_h \|_{L^p(0,t; L^q(\Gamma_h(s)))} \notag\\
	&\quad
	+ C \| \nabla_{\Gamma_h(s)} u_h \|_{L^p(0,t; L^q(\Gamma_h(s)))} 
	+ C |s-t| \| \Delta_{\Gamma_h(s),h} u_h \|_{L^p(0,t; L^q(\Gamma_h(s)))} , \notag
\end{align}
where in the second inequality we have used Lemma \ref{lemma:delta_h}. Then we apply the norm equivalence, Remark \ref{rmk:delta_h} and Lemma \ref{lemma:interp} to the inequality above to change the domain from $\Gahs$ to $\Gah(0)$:
\begin{align}\label{eq:scheme_pb}
	&\quad\| \partial_t u_h \|_{L^p(0,t; L^q(\Gamma_h(0)))} 
	+ 
	\| \Delta_{\Gamma_h(s),h} u_h \|_{L^p(0,t; L^{q}(\Gamma_h(0)))} \notag\\
	&\leq
	C \| f_h \|_{L^p(0,t; L^q(\Gamma_h(0)))}
	+
	C \| u_h \|_{L^p(0,t; L^q(\Gamma_h(0)))}
	+
	C |s-t|\| \partial_t u_h \|_{L^p(0,t; L^q(\Gamma_h(0)))} \notag\\
	&\quad
	+ C |s-t| \| \Delta_{\Gamma_h(0),h} u_h \|_{L^p(0,t; L^q(\Gamma_h(0)))} .
\end{align}

We define (cf.~the proof of Theorem~3.1 and (4.22) in \cite{KL23})
\begin{align}
	L_{p,q}(t) := 
	\| \partial_t u_h \|_{L^p(0,t; L^q(\Gamma_h(0)))}^p
	+
	\| \Delta_{\Gamma_h(0),h} u_h \|_{L^p(0,t; L^q(\Gamma_h(0)))}^p \notag
\end{align}
with
\begin{align}
	\partial_t L_{p,q}(t) = 
	\| \partial_t u_h(t,\cdot) \|_{L^q(\Gamma_h(0)))}^p
	+
	\| \Delta_{\Gamma_h(0),h} u_h(t,\cdot) \|_{L^q(\Gamma_h(0)))}^p . \notag
\end{align}
Then \eqref{eq:scheme_pb} implies that
\begin{align}
	L_{p,q}(s)
	&=
	\| \partial_t u_h \|_{L^p(0,s; L^q(\Gamma_h(0)))}^p
	+
	\| \Delta_{\Gamma_h(0),h} u_h \|_{L^q(0,s; L^2(\Gamma_h(0)))}^p
	\notag\\
	&\leq
	C \| f_h \|_{L^p(0,s; L^q(\Gamma_h(0)))}^p
	+
	C \| u_h \|_{L^p(0,s; L^q(\Gamma_h(0)))}^p
	\notag\\
	&\quad+
	C \int_0^s (s-t)^p \big( \| \partial_t u_h(t,\cdot) \|_{L^q(\Gamma_h(0))}^p
	+
	\| \Delta_{\Gamma_h(0),h} u_h(t,\cdot) \|_{L^q(\Gamma_h(0))}^p \big) \d t
	\notag\\
	&\leq
	C \| f_h \|_{L^p(0,s; L^q(\Gamma_h(0)))}^p
	+
	C T^p \int_0^s L_{p,q}(t) \d t
	+
	C p T^{p-1} \int_0^s L_{p,q}(t) \d t , \notag
\end{align}
where we have used the following estimates in the last line (cf.~\cite[p. 10]{KL23})
\begin{align}
	&\| u_h \|_{L^p(0,s; L^q(\Gamma_h(0)))}^p
	\leq
	T^p \int_0^s L_{p,q}(t) \d t , \notag\\
	&
	\int_0^s (s-t)^p \big( \| \partial_t u_h(t,\cdot) \|_{L^q(\Gamma_h(0))}^p
	+
	\| \Delta_{\Gamma_h(0),h} u_h(t,\cdot) \|_{L^q(\Gamma_h(0))}^p \big) \d t
	\leq
	p T^{p-1} \int_0^s L_{p,q}(t) \d t  . \notag
\end{align}
Thus, by Gr\"onwall's inequality, it holds that
\begin{align}
	L_{p,q}(s)
	\leq
	C \| f_h \|_{L^p(0,s; L^q(\Gamma_h(0)))}^p  \quad\forall s\in [0,T] . \notag
\end{align}

Therefore, the proof of Theorem \ref{thm:disc_max_reg} is complete.

\section*{Acknowledgements}

The work of G. Bai was was partially supported by the Hong Kong Research Grants Council (Project No.~PolyU/15303022). 

The work of B.~Kov\'acs is funded by the Heisenberg Programme of the Deutsche Forschungsgemeinschaft (DFG, German Research Foundation) -- Project-ID 446431602, and by the DFG Research Unit FOR 3013 \textit{Vector- and tensor-valued surface PDEs} (BA2268/6–1).

The work of B. Li was partially supported by a fellowship award from the Research Grants Council of the Hong Kong Special Administrative Region, China (Project No.~PolyU/RFS2324-5S03).

\appendix

\section{Proof of Lemma \ref{LocEEst}}
\label{appendix:A}
\setcounter{equation}{0}

The following lemma is proved in \cite[Lemma A.1]{Li-MCOM-2019}, which can be easily extended to our current scenario.
\begin{lemma}\label{lemma:loc_est_2}
	Suppose that, for any $t \in (0,T)$, $\phi_h^\ell(t) \in S_h(\Gamma)$ satisfies
	\begin{align}
		&(\partial_t \phi_h^\ell, \chi_h^\ell)_\Gamma
		+
		(\nabla_\Gamma \phi_h^\ell, \nabla_\Gamma \chi_h^\ell)_\Gamma
		=
		0
		\quad\forall \chi_h^\ell \in S_h^0(\Gamma_j''), t\in (0, d_j^2) , \notag \\
		&(\partial_t \phi_h^\ell, \chi_h^\ell)_\Gamma
		+
		(\nabla_\Gamma \phi_h^\ell, \nabla_\Gamma \chi_h^\ell)_\Gamma
		=
		0
		\quad\forall \chi_h^\ell \in S_h^0(D_j''), t\in (d_j^2 / 4, 2 d_j^2) . \notag
	\end{align}
	Then we have
	\begin{align}
		&\vertiii{\partial_t \phi_h^\ell}_{Q_j}
		+
		d_j^{-1} \vertiii{\nabla_\Gamma \phi_h^\ell}_{1,Q_j}
		\notag\\
		&\leq
		(C h^{1/2} d_j^{-1/2} + C\epsilon^{-1} h d_j^{-1} + \epsilon)
		\big(\vertiii{\partial_t \phi_h^\ell}_{Q_j'}
		+
		d_j^{-1} \vertiii{\nabla_\Gamma \phi_h^\ell}_{1,Q_j'}\big)
		\notag\\
		&\quad+
		C \epsilon^{-1}
		\bigg(d_j^{-2}\vertiii{\phi_h^\ell}_{Q_j'}
		+
		d_j^{-1} \| \phi_h^\ell(0) \|_{\Gamma_j'}
		+
		\| \nabla_\Gamma \phi_h^\ell(0) \|_{1,\Gamma_j'}
		\bigg) , \notag
	\end{align}
	where the constant $C>0$ is independent of $h$, $j$ and $C_*$.
\end{lemma}

It is easy the construct the space-time cut-off function $\tilde\omega$ which is constant one on $Q_j'$ and zero outside $Q_j''$. We then define $\tilde\phi = \tilde\omega \phi$ and $\tilde\phi_h^\ell \in S_h(\Gamma)$ to be the finite element solution of 
\begin{align}\label{eq:err_eq_2}
	(\partial_t \tilde\phi - a_h\partial_t \tilde\phi_h^\ell, \chi_h^\ell)_\Gamma
	+ (\nabla_\Gamma \tilde\phi - B_h\nabla_\Gamma \tilde\phi_h^\ell, \chi_h^\ell)_\Gamma = 0 \quad\forall \chi_h^\ell \in S_h(\Gamma),
\end{align}
with the initial condition $\tilde\phi_h^\ell(0) = \tilde\phi(0)  = 0$. Obviously, $\tilde \phi_h^\ell$ is zero outside $Q_j''$ as well.

Then,
\begin{align}
	&\quad (a_h \partial_t (I_h\tilde\phi - \tilde\phi_h^\ell), \chi_h^\ell)_\Gamma
	+ (B_h \nabla_\Gamma (I_h \tilde\phi - \tilde\phi_h^\ell), \chi_h^\ell)_\Gamma \notag\\
	&=
	(a_h \partial_t (I_h\tilde\phi - \tilde\phi), \chi_h^\ell)_\Gamma
	+ (B_h \nabla_\Gamma (I_h \tilde\phi - \tilde\phi), \chi_h^\ell)_\Gamma \notag\\
	&\quad+
	((a_h - 1)\partial_t \tilde\phi, \chi_h^\ell)_\Gamma
	+ ((B_h - 1)\nabla_\Gamma \tilde\phi, \chi_h^\ell)_\Gamma \quad\forall \chi_h^\ell\in S_h(\Gamma) . \notag
\end{align}
Testing the above equation with $\chi_h^\ell = I_h\tilde\phi - \tilde\phi_h^\ell \in S_h(\Gamma)$ and using Lemma \ref{lemma:geo_perturb}, temporal integration by parts and H\"older's inequality (cf.~\cite[p. 37]{Li-MCOM-2019}), we get
\begin{align}\label{eq:energy_est1}
	&\quad \| I_h \tilde\phi - \tilde\phi_h^\ell \|_{L^\infty L^2(\mathcal Q)}^2
	+ \|\nabla_\Gamma ( I_h \tilde\phi - \tilde\phi_h^\ell) \|_{L^2 L^2(\mathcal Q)}^2 \notag\\
	&\leq 
	\| I_h \tilde\phi - \tilde\phi_h^\ell \|_{L^2 L^2(\mathcal Q)} \|\partial_t ( I_h \tilde\phi - \tilde\phi) \|_{L^2 L^2(\mathcal Q)}
	+ \| \nabla_\Gamma (I_h \tilde\phi - \tilde\phi) \|_{L^2 L^2(\mathcal Q)} \|\nabla_\Gamma ( I_h \tilde\phi - \tilde\phi_h^\ell) \|_{L^2 L^2(\mathcal Q)} \notag\\
	&\quad
	+ C h^{k+1} \| \tilde\phi \|_{L^2 L^2(\mathcal Q)} \|\partial_t ( I_h \tilde\phi - \tilde\phi_h^\ell) \|_{L^2 L^2(\mathcal Q)}
	+ C h^{k+1} \| \nabla_\Gamma  \tilde\phi \|_{L^2 L^2(\mathcal Q)} \|\nabla_\Gamma ( I_h \tilde\phi - \tilde\phi_h^\ell) \|_{L^2 L^2(\mathcal Q)} ,
\end{align} 
and, similarly, testing the above equation with $\chi_h^\ell = \partial_t (I_h\tilde\phi - \tilde\phi_h^\ell) \in S_h(\Gamma)$ leads to
\begin{align}\label{eq:energy_est2}
	&\quad \| \partial_t (I_h \tilde\phi - \tilde\phi_h^\ell) \|_{L^2 L^2(\mathcal Q)}^2
	+ \|\nabla_\Gamma ( I_h \tilde\phi - \tilde\phi_h^\ell) \|_{L^\infty L^2(\mathcal Q)}^2 \notag\\
	&\leq 
	\| \partial_t (I_h \tilde\phi - \tilde\phi) \|_{L^2 L^2(\mathcal Q)} \|\partial_t ( I_h \tilde\phi - \tilde\phi_h^\ell) \|_{L^2 L^2(\mathcal Q)}
	+ \| \nabla_\Gamma \partial_t (I_h \tilde\phi - \tilde\phi) \|_{L^2 L^2(\mathcal Q)} \|\nabla_\Gamma ( I_h \tilde\phi - \tilde\phi_h^\ell) \|_{L^2 L^2(\mathcal Q)} \notag\\
	&\quad
	+ C h^{k+1} \| \partial_t \tilde\phi \|_{L^2 L^2(\mathcal Q)} \|\partial_t ( I_h \tilde\phi - \tilde\phi_h^\ell) \|_{L^2 L^2(\mathcal Q)}
	+ C h^{k+1} \| \nabla_\Gamma \partial_t \tilde\phi \|_{L^2 L^2(\mathcal Q)} \|\nabla_\Gamma ( I_h \tilde\phi - \tilde\phi_h^\ell) \|_{L^2 L^2(\mathcal Q)} .
\end{align}
Applying Young's inequality to \eqref{eq:energy_est1} and \eqref{eq:energy_est2}, we obtain
\begin{align}
	 &\quad \| \partial_t (I_h \tilde\phi - \tilde\phi_h^\ell) \|_{L^2 L^2(\mathcal Q)}
	 + \|\nabla_\Gamma ( I_h \tilde\phi - \tilde\phi_h^\ell) \|_{L^\infty L^2(\mathcal Q)} \notag\\
	 &\quad+ d_j^{-1} \| I_h \tilde\phi - \tilde\phi_h^\ell \|_{L^\infty L^2(\mathcal Q)}
	 + d_j^{-1} \|\nabla_\Gamma ( I_h \tilde\phi - \tilde\phi_h^\ell) \|_{L^2 L^2(\mathcal Q)} \notag\\
	 &\leq 
	 C \| \partial_t (I_h \tilde\phi - \tilde\phi) \|_{L^2 L^2(\mathcal Q)}
	 + C d_j \| \nabla_\Gamma \partial_t (I_h \tilde\phi - \tilde\phi) \|_{L^2 L^2(\mathcal Q)} \notag\\
	 &\quad
	 + C d_j^{-2} \| I_h \tilde\phi - \tilde\phi \|_{L^2 L^2(\mathcal Q)}
	 + C d_j^{-1} \| \nabla_\Gamma (I_h \tilde\phi - \tilde\phi) \|_{L^2 L^2(\mathcal Q)} \notag\\
	 &\quad
	 + C h^{k+1} \| \partial_t \tilde\phi \|_{L^2 L^2(\mathcal Q)}
	 + C d_j h^{k+1} \| \nabla_\Gamma \partial_t \tilde\phi \|_{L^2 L^2(\mathcal Q)} \notag\\
	 &\quad
	 + C d_j^{-2} h^{k+1} \| \tilde\phi \|_{L^2 L^2(\mathcal Q)}
	 + C d_j^{-1} h^{k+1} \| \nabla_\Gamma \tilde\phi \|_{L^2 L^2(\mathcal Q)} . \notag
\end{align}
Using the triangle inequality and H\"older's inequality in the temporal direction,
\begin{align}\label{eq:energy_est3}
	&\quad \| \partial_t (\tilde\phi - \tilde\phi_h^\ell) \|_{L^2 L^2(\mathcal Q)}
	+ d_j^{-1} \|\nabla_\Gamma (\tilde\phi - \tilde\phi_h^\ell) \|_{L^2 L^2(\mathcal Q)} \notag\\
	&\quad+ d_j^{-2} \| \tilde\phi - \tilde\phi_h^\ell \|_{L^2 L^2(\mathcal Q)}
	+ d_j^{-1} \|\nabla_\Gamma ( \tilde\phi - \tilde\phi_h^\ell) \|_{L^2 L^2(\mathcal Q)} \notag\\
	&\leq 
	C \| \partial_t (I_h \tilde\phi - \tilde\phi) \|_{L^2 L^2(\mathcal Q)}
	+ C d_j \| \nabla_\Gamma \partial_t (I_h \tilde\phi - \tilde\phi) \|_{L^2 L^2(\mathcal Q)} \notag\\
	&\quad
	+ C d_j^{-2} \| I_h \tilde\phi - \tilde\phi \|_{L^2 L^2(\mathcal Q)}
	+ C d_j^{-1} \| \nabla_\Gamma (I_h \tilde\phi - \tilde\phi) \|_{L^2 L^2(\mathcal Q)} \notag\\
	&\quad
	+ C h^{k+1} \| \partial_t \tilde\phi \|_{L^2 L^2(\mathcal Q)}
	+ C d_j h^{k+1} \| \nabla_\Gamma \partial_t \tilde\phi \|_{L^2 L^2(\mathcal Q)} \notag\\
	&\quad
	+ C d_j^{-2} h^{k+1} \| \tilde\phi \|_{L^2 L^2(\mathcal Q)}
	+ C d_j^{-1} h^{k+1} \| \nabla_\Gamma \tilde\phi \|_{L^2 L^2(\mathcal Q)} .
\end{align}

By subtracting \eqref{eq:err_eq_1} from \eqref{eq:err_eq_2} and using suitably arranged cut-off function, we know that for $\chi_h^\ell \in S_h^0(\Gamma_j''), t\in (0, d_j^2)$ and $\chi_h^\ell \in S_h^0(D_j''), t\in (d_j^2 / 4, 2 d_j^2)$
\begin{align}
	(a_h\partial_t(\tilde\phi_h^\ell - \phi_h^\ell),\chi_h^\ell)_\Gamma
	+(B_h\nabla_\Gamma (\tilde\phi_h^\ell - \phi_h^\ell),\nabla_\Gamma \chi_h^\ell)_\Gamma =0.  \notag
\end{align} 
Then we apply Lemma \ref{lemma:loc_est_2} to $\tilde\phi_h - \phi_h$ and get
\begin{align}\label{eq:energy_est4}
	&\vertiii{\partial_t (\tilde\phi_h^\ell - \phi_h^\ell)}_{Q_j}
	+
	d_j^{-1} \vertiii{\tilde\phi_h^\ell - \phi_h^\ell}_{1,Q_j}
	\notag\\
	&\leq
	(C h^{1/2} d_j^{-1/2} + C\epsilon^{-1} h d_j^{-1} + \epsilon)
	\bigg(\vertiii{\partial_t (\tilde\phi_h^\ell - \phi_h^\ell)}_{Q_j'}
	+
	d_j^{-1} \vertiii{\tilde\phi_h^\ell - \phi_h^\ell}_{1,Q_j'}\bigg)
	\notag\\
	&\quad+
	C \epsilon^{-1}
	\bigg(d_j^{-2}\vertiii{\tilde\phi_h^\ell - \phi_h^\ell}_{Q_j'}
	+
	d_j^{-1} \| \phi_h^\ell(0) \|_{Q_j'}
	+
	\| \phi_h^\ell(0) \|_{1,Q_j'}
	\bigg) \notag\\
	&\leq
	(C h^{1/2} d_j^{-1/2} + C\epsilon^{-1} h d_j^{-1} + \epsilon)
	\bigg(\vertiii{\partial_t (\tilde\phi - \phi_h^\ell)}_{Q_j'}
	+
	d_j^{-1} \vertiii{\tilde\phi - \phi_h^\ell}_{1,Q_j'}\bigg)
	\notag\\
	&\quad+
	C \epsilon^{-1}
	\bigg(d_j^{-2}\vertiii{\tilde\phi - \phi_h^\ell}_{Q_j'}
	+
	d_j^{-1} \| \phi_h^\ell(0) \|_{Q_j'}
	+
	\| \phi_h^\ell(0) \|_{1,Q_j'}
	\bigg) \notag\\
	&\quad+
	(C h^{1/2} d_j^{-1/2} + C\epsilon^{-1} h d_j^{-1} + \epsilon)
	\bigg(\vertiii{\partial_t (\tilde\phi_h^\ell - \tilde\phi)}_{Q_j'}
	+
	d_j^{-1} \vertiii{\tilde\phi_h^\ell - \tilde\phi}_{1,Q_j'}\bigg)
	\notag\\
	&\quad+
	C \epsilon^{-1}
	d_j^{-2}\vertiii{\tilde\phi_h^\ell - \tilde\phi}_{Q_j'} ,
\end{align}
where in the last inequality, we use the splitting $\tilde\phi - \phi_h^\ell = (\tilde\phi - \tilde\phi_h^\ell) + (\tilde\phi_h^\ell - \phi_h^\ell)$.

Finally, we complete the proof of Lemma \ref{LocEEst} by applying the triangle inequality to \eqref{eq:energy_est3} and \eqref{eq:energy_est4} and using the fact that both $\tilde\phi$ and $\tilde \phi_h^\ell$ are zero outside $Q_j''$.

\medskip

\def\baselinestretch{0.95}
{\small
\bibliographystyle{abbrv}
\bibliography{WMP}

}
\end{document}